\documentclass[11pt]{amsart}
\usepackage{amscd,amssymb}
\theoremstyle{plain}

\newtheorem{thm}[equation]{Theorem}
\newtheorem{prop}[equation]{Proposition}
\newtheorem{cor}[equation]{Corollary}

\newtheorem{lemma}[equation]{Lemma}

\numberwithin{equation}{section}

\def\Hom{{\mathrm{Hom}}}

\def\Aut{{\mathrm{Aut}}}

\def\SL{{\mathrm{SL}}}

\def\Sp{{\mathrm{Sp}}}
\def\Spin{{\rm Spin}}

\def\SU{{\mathrm {SU}}}
\def\U{{\mathrm U}}

\def\GL{{\mathrm{GL}}}
\def\PGL{{\rm PGL}}

\def\SO{{\rm{SO}}}
\def\Ind{{\mathrm{Ind}}}

\def\O{{\mathrm O}}

\begin{document}
\title{A family of $\mathrm{Spin}(8)$ dual pairs: the case of real groups}

\begin{abstract} 
Exceptional groups of type $E_6$ contain dual pairs where one member is $\mathrm{Spin}(8)$, and the other is 
$T\rtimes  \mathbb Z/2\mathbb Z$, where $T$ is a two-dimensional torus and 
the non-trivial element in $\mathbb Z/2\mathbb Z$ acts on $T$ by the inverse involution. 
We describe the correspondence of 
representations arising by restricting the minimal representation.  
\end{abstract} 


\author{Wee Teck Gan} 
\address{Department of Mathematics, National University of Singapore, Singapore}
\email{matgwt@nus.edu.sg}

\author{Hung Yean Loke} 
\address{Department of Mathematics, National University of Singapore, Singapore}
\email{matlhy@nus.edu.sg}

\author{Annegret Paul} 
\address{Department of Mathematics, Western Michigan University, Kalamazoo, MI}
\email{annegret.paul@wmich.edu}

\author{Gordan Savin}
\address{Department of Mathematics, University of Utah, Salt Lake City, UT}

\email{savin@math.utah.edu}

\maketitle

\section{Introduction}  
Let $\mathfrak h$ be a simple, split Lie algebra of type $F_4, E_6, E_7$ or $E_8$, over a field $F$ of characteristic 0.  Let $H$ be the 
group of automorphism of $\mathfrak h$. A feature of exceptional groups is that $H$ contains a dual reductive pair 
\[ 
G_2 \times \Aut(J_s) \subset H 
\] 
where $G_2$ is the split exceptional split group of type $G_2$, and $\Aut(J_s)$ is the group of automorphisms of a split Freudenthal-Jordan algebra  $J_s$ of dimension 
6, 9, 15 and 27, respectively. Typical examples of such Jordan algebras, not necessarily split, are the algebras of $3\times 3$ hermitian symmetric matrices with coefficients 
in a composition algebra of dimension 1, 2, 4, and 8, respectively. Over real numbers typical examples of composition algebras are $\mathbb R$, $\mathbb C$, 
$\mathbb H$ and $\mathbb O$, the algebra of octonions.  
For a precise definition of these algebras the reader may consult {\it The Book of Involutions} \cite{KMRT}. 
 The group $G_2$ is itself the group of automorphisms of the split octonion algebra $\mathbb O_s$. 
By functoriality of Galois cohomology we have a natural map 
 \[ 
 H^1(F, G_2) \times H^1(F, \Aut(J_s))  \rightarrow H^1(F, H). 
 \] 
Thus a  form of $\mathbb O_s$ and a form of $J_s$ (a Freudenthal-Jordan algebra $J$) give rise to a form of $H$. 
This is the Tits construction from the Galois cohomology point of view. In this paper we shall keep $\mathbb O_s$ and vary $J$. The resulting form of $H$ is 
denoted by $H_J$. It contains a dual reductive pair (with $G_2$ split) 
\[ 
G_2 \times \Aut(J) \subset H_J. 
\] 

\smallskip 
Cubic etal\'e algebras over $F$ naturally appear as subalgebras of $J$. Indeed, 
every element in $J$ satisfies its characteristic polynomial. For Freudenthal-Jordan algebras the characteristic polynomials are cubic. Thus a generic element in $J$ 
generates a cubic etal\'e subalgebra. 
 Let $i: E\rightarrow J$ be an embedding of a cubic etal\'e algebra $E$ into $J$. Let $\Aut(E\rightarrow J)$ be the 
subgroup of  $\Aut(J)$ consisting of all automorphisms of $J$ fixing $i(E)$ pointwise. The centralizer of $\Aut(E\rightarrow J)$  in $H_J$ is $G_E$, 
a quasi-split group of absolute type $D_4$ corresponding to $E$. Let $Z_E$ be the center of $G_E$. Then 
\[
G_E \times_{Z_E}\Aut(E\rightarrow J)\subset H_J 
\] 
is a dual reductive pair in $H_J$, in a see-saw position relative to $G_2 \times \Aut(J)$. 

\smallskip 

In this paper we shall exclusively be interested in the case when the type of 
$H$ is $E_6$. So it is perhaps a good time to describe the groups involved in more details. 
In this case the split algebra $J_s$ is $M_3(F)$, the algebra of $3\times 3$ matrices with coefficients in $F$, and 
$\Aut(J)=\PGL_3\rtimes \mathbb Z/2\mathbb Z$, where $\PGL_3$ acts by conjugation and 
  the factor $\mathbb Z/2\mathbb Z$ by the transpose map. 
 Let $E_s=F^3$, the split cubic etal\'e algebra, and let $i: E_s \rightarrow J_s$ be the embedding given by the diagonal matrices. 
 It clear that $\Aut(E_s\rightarrow J_s)$ is $T\rtimes  \mathbb Z/2\mathbb Z$, where $T$ is the maximal torus in $\PGL_3$, and 
 the conjugation action of the non-trivial element in $\mathbb Z/2\mathbb Z$ on $T$ is the inverse involution of $T$. The second dual pair is not a maximal 
 subgroup in $H_J$. Indeed, the normalizer of $T$ in $\PGL_3$ contains $S_3$, the group of permutations of three letters, and this group acts simultaneously on 
 $T\rtimes  \mathbb Z/2\mathbb Z$ and on its centralizer, a split simply connected group of type $D_4$, by outer automorphisms. 
 
 \smallskip 
 The map of Galois cohomology 
 \[ 
 H^1(F, \PGL_3 \rtimes \mathbb Z/2\mathbb Z) \rightarrow H^1(F, \mathbb Z/2\mathbb Z) = \{\text{etal\'e quadratic $F$-algebras}\} . 
 \] 
 attaches a quadratic etal\'e algebra $K=K(J)$ to any $F$-form  $J$ of $M_3(F)$.  If $F=\mathbb R$ then 
  the group $\Aut(E\rightarrow J)$ depends on $E$ and $K$ only \cite{GS14}. Thus 
 we shall adopt the following notation: 
 \[ 
 T_{E,K}:=\Aut(E\rightarrow J).
 \]

 The goal of this paper is to compute the correspondence of representations obtained by restricting the minimal representation $V$ of $H_J$ to the dual 
 pair  $G_E \times T_{E,K}$ for real groups, i.e. when $F=\mathbb R$. These results were announced in \cite{GS22} where the case of $p$-adic $F$ was treated. 
  If $F=\mathbb R$ then the cubic etal\'e algebras are $\mathbb R^3$ and $\mathbb R \times \mathbb C$. 
 The Jordan algebras are  $M_3(\mathbb R)$ and $H_3(\mathbb C)$, the algebra of hermitian symmetric matrices in $M_3(\mathbb C)$ with respect to an involution of the 
 second kind on $M_3(\mathbb C)$.  The quadratic etal\'e algebra $K=K(J)$ is $\mathbb R^2$ and $\mathbb C$, respectively. 
 Thus there are essentially 4 cases to consider, as we vary $E$ and $K$.  In each case $T_{E,K}$  is a semi-direct product of  
 the two-dimensional torus $T^{\circ}_{E,K}$ and $\mathbb Z/2\mathbb Z$. 
 Thus almost all irreducible representations of $T_{E,K}$ are 2-dimensional and the restriction to $T_{E,K}^{\circ}$ 
  is a direct sum $\chi\oplus\chi^{-1}$ of two characters. Such a representation of $T_{E,K}$ is denoted by $\rho(\chi)$. 
 Then there are two representations trivial on $T_{E,K}^{\circ}$, the trivial denoted by $\rho(1)$, and the non-trivial, denoted by $\epsilon$.  
 For every unitary representation $\rho$  of  $T_{E,K}$  we compute the lift $\Theta(\rho)$ to $G_E$. 
 It is always non-zero and irreducible, and we compute its Langlands parameter.   
 The answer is uniform in $\chi$ except, of course, for the lift of $\epsilon$, which is the most difficult case. 
 
 \smallskip 
 This is accomplished as follows. We firstly determine the infinitesimal character of $\Theta(\rho)$.  
Secondly, as a simple consequence of branching rules, the $K$-types of $\Theta(\rho)$ are contained in a cone. This severely limits the possible 
 irreducible subquotients of $\Theta(\rho)$, when we compare with the $K$-types 
 of irreducible representations of $G_E$ with the same infinitesimal character. For this comparison we only need to know multiplicities of finitely 
  many $K$-types, and this information is obtained using the {\tt atlas} software \cite{A}, when dealing with $\rho=1$ and $\epsilon$.
   Thirdly, as the key additional input, 
  we prove that certain $K$-types appear with multiplicity one in $\Theta(\rho)$ and this implies irreducibility of 
 $\Theta(\rho)$. 
 

\section{Groups}

\subsection{Quasi-split groups of type $D_4$} 
 Let $G$ be the simply connected group of Lie type $D_4$ over a field $F$ of characteristic 0. This group is exceptional 
 in the sense that its Dynkin diagram has a group of automorphisms isomorphic to $S_3$, the symmetric group of permutations of three non-branching vertices. 
 
 
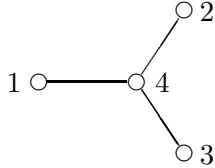
\begin{picture}(400,110)(110,25)

\put(321,73){\line(2,-3){14}}

\put(270,72){1}

\put(282,76){\circle{6}}
\put(285,76){\line(1,0){30}}
\put(319,76){\circle{6}}

\put(326,72){4}

\put(337,103){\circle{6}}
\put(337,49){\circle{6}}

\put(321,73){\line(2,-3){14}}
\put(321,79){\line(2,3){14}}

\put(343,99){2}
\put(343,45){3}

\end{picture}

We realize the root system in the standard way in $\mathbb R^4$ such that the set of simple roots is 
 $\{ \alpha_1 = \varepsilon_1 - \varepsilon_2, \alpha_4 = \varepsilon_2 - \varepsilon_3,,  \alpha_2 = \varepsilon_3 - \varepsilon_4, \alpha_3 = \varepsilon_3 + \varepsilon_4\}$ 
 with the labeling as in the picture. 
The highest root is $\beta=\varepsilon_1 + \varepsilon_2 = (1,1,0,0)$.

\smallskip 
A cubic etal\'e algebras $E$ over $F$ is an $F$-form of the split cubic etal\'e algebra $E_s= F^3$. The group of $F$-automorphisms of $E_s$ is also $S_3$ thus the 
isomorphism classes of quasi-split forms of $G$ correspond to  the isomorphism classes of cubic etal\'e algebras $E$ over $F$.  
Let $G_E$ be the group of $F$-points of the quasi-split form of $G$ corresponding to $E$.  The group $G_E$ has a maximal parabolic $P_E=M_EN_E$ corresponding to 
the branching vertex of the Dynkin diagram. The Levi factor $M_E$ is isomorphic to 
\[ 
\GL_2(E)^{\det} =\{ g\in \GL_2(E) | \det(g)\in F^{\times} \}. 
\] 
The nilpotent radical $N_E$ is a 9-dimensional 2-step unipotent group with one dimensional center $[N_E,N_E]$. The group $M_E$ acts on the center of $N_E$ 
 by the determinant character. Let $V_E=N_E/[N_E,N_E]$, and let $V$  be the standard 2-dimensional representation of $\GL_2(\bar F)$. 
 Then $V_E(\bar F) \cong V \otimes V\otimes V$, as $M_E(\bar F)\cong \GL_2(\bar F^3)^{\det}$-modules, where the standard 
action of $\GL_2(\bar F^3)^{\det}$ on $V \otimes V\otimes V$ is twisted by $\det^{-1}$. Hence it is easy to see that 
the modular character $\rho_{N_E}$ of $M_E$ is 
\[ 
\rho_{N_E}=|\det|^5. 
\] 
Let $\pi$ be a tempered representation of $M_E$. Using the normalized induction 
induce $\pi \otimes |\det|^s$ from $P_E$ to $G_E$.  If $s>0$, this is a standard module. 
Let $L(\pi, s)$ be the corresponding Langlands quotient. 


 \subsection{Real groups $T_{E,K}$ }  
 Assume $F=\mathbb R$.  Every real Freudenthal-Jordan algebra $J$ of dimension 9 is isomorphic to $M_3(\mathbb R)$ or  $H_3(\mathbb C)$, 
 the set of fixed points of an involution of the second kind on $M_3(\mathbb C)$.  
 Recall the quadratic etal\'e algebra $K=K(J)$ which arises from the map of Galois cohomology 
 \[ 
 H^1(F, \PGL_3 \rtimes \mathbb Z/2\mathbb Z) \rightarrow H^1(F, \mathbb Z/2\mathbb Z) = \{\text{etal\'e quadratic $F$-algebras}\} . 
 \] 
Over  $\mathbb R$,  there are two  quadratic etal\'e algebras: $\mathbb R^2$ and $\mathbb C$. If 
$J=M_3(\mathbb R)$ then $K=\mathbb R^2$. 
If $J=H_3(\mathbb C)$, then $K=\mathbb C$. 
 The group $H_J$ depends only on $K$. It is split if $K=\mathbb R^2$, and quasi-split if $K=\mathbb C$ \cite{LS15} .

 \smallskip 
 Our next task is to describe the possible embeddings. In the real case there are 2 cubic etal\'e algebras: 
 $E=\mathbb R^3$ and $E=\mathbb R \times \mathbb C$. If $J$ is isomorphic to 
$M_3(\mathbb R)$ then both $E$ embed into $M_3(\mathbb R)$ and these embeddings are unique up to conjugation. So assume that $J$ is isomorphic to 
 $ H_3(\mathbb C)$, the set of fixed points of an involution of the second kind on $M_3(\mathbb C)$. 
 The involution of the second kind  on $M_3(\mathbb C)$ on arises from a non-degenerate hermitian-symmetric form $h$ on $\mathbb C^3$. Let 
 $h=\epsilon_1 z_1 \bar z_1 + \epsilon_2 z_2 \bar z_2 + \epsilon_3 z_3 \bar z_3$, where $\epsilon_i=\pm 1$. 
 We have 8 choices for signs, but we get 4 different involutions since $h$ and $-h$ give the same involution. In this way we get 4 Jordan algebras $J$  together with 
 an embedding of $\mathbb R^3$ into $J$  as diagonal matrices. 
 One of this Jordan algebras is definite, i.e. $h$ is definite, while three other are indefinite and isomorphic to each other.
  However the three embeddings are not isomorphic, and we get 4 classes of embeddings in all. 
 These represent all isomorphism classes of embeddings of $E=\mathbb R^3$ into such Jordan algebras. Finally, $E=\mathbb R\times \mathbb C$ does not embed into the 
 definite Jordan algebra and there is a unique embedding of this $E$ into an indefinite Jordan algebra. For each embedding it is easy to compute 
 $T_{E,K}$. It is always a semi-direct product $T_{E,K}^{\circ} \rtimes \mathbb Z/2\mathbb Z$ such that the conjugation action of the non-trivial element in 
 $\mathbb Z/2\mathbb Z$ on $T^{\circ}_{E,K}$ is the inverse involution. 
  The possible cases of the two-dimensional torus $T^{\circ}_{E,K}$  are tabulated in the following table, where $\mathbb T$ the group of complex numbers of norm one. 
\[
\begin{array}{c||c|c} 
& E=\mathbb R^3 & E=\mathbb R \times \mathbb C \\
\hline 
K=\mathbb R^2 & (\mathbb R^{\times })^3/\Delta(\mathbb R^{\times}) & (\mathbb R^{\times } \times \mathbb C^{\times}) /\Delta(\mathbb R^{\times}) \\ 
\hline 
K=\mathbb C &  (\mathbb T)^3/\Delta(\mathbb T) & (\mathbb T \times \mathbb C^{\times}) /\Delta(\mathbb T)
\end{array} 
\]

\smallskip 
We introduce a refined notation for characters of these tori. A character $\chi$ of  $(\mathbb R^{\times })^3/\Delta(\mathbb R^{\times})$ is a triple of 
characters $(\chi_1, \chi_2, \chi_3)$ of $\mathbb R^{\times}$ such that $\chi_1\cdot \chi_2\cdot \chi_3=1$. A character of $\mathbb T$ is represented by an integer. Thus 
a character $\chi$ of  $(\mathbb T)^3/\Delta(\mathbb T)$ is represented by a triple of integers $(n_1, n_2, n_3)$ such that $n_1 + n_2 + n_3=0$. In the remaining two cases 
a character of the torus is identified with a pair of characters $(\chi_{\mathbb R}, \chi_{\mathbb C})$, such that  $\chi_{\mathbb R}\cdot \chi_{\mathbb C}=1$ on  
$\Delta(\mathbb R^{\times})$ ,
 and with a pair $(m, \chi_{\mathbb C})$, where $m \in \mathbb Z$, such that the restriction of $\chi_{\mathbb C}$ to $\mathbb T$  is given 
by $z\mapsto z^{-m}$.

\section{Main statement}  

\subsection{Representations of $T_{E,K}$}  
Let $\chi$ be a character of $T^{\circ}_{E,K}$. If $\chi\neq \chi^{-1}$, let $\rho(\chi)\cong \rho(\chi^{-1})$ be the unique irreducible representation of 
$T_{E,K}$ such that the restriction to 
$T^{\circ}_{E,K}$ is $\chi\oplus \chi^{-1}$. If $\chi=\chi^{-1}$, then $\chi$ extends to a character of $T_{E,K}$, in two ways, denoted by $\rho(\chi)^{\pm}$. 
These two representations 
are indistinguishable unless $\chi=1$. Then one extension is the trivial representation, denoted by $\rho(1)$, and the other the sign representation $\epsilon$.  
Note that non-trivial quadratic characters $\chi$ appear only in the split case.

\subsection{Some tempered representations of $M_E$}  \label{packets} 
 To every unitary character $\chi$ of $T_{E,K}$ we shall attach 
a packet $P(E,K,\chi)=P(E,K,\chi^{-1})$ of tempered representations of $\GL_2(E)^{\det}$, 
obtained by restricting an irreducible representation of $\GL_2(E)$. 
We need additional notation. Let $F$ be a local field, and $(\mu_1,\mu_2)$ a pair of characters of  $F^{\times}$.  
Let $\mu_1\times \mu_2$ be the unique infinite-dimensional sub-quotient of the principal series representation of $\GL_2(F)$
obtained by normalized parabolic induction from the pair of characters. 
Let $\omega : \mathbb R^{\times} \rightarrow \{\pm 1\}$ be the sign character. It is the unique non-trivial quadratic character of $\mathbb R^{\times}$. 
Let $\nu : \mathbb R^{\times} \rightarrow \mathbb R^{\times}$ be $\nu(x)=x$, for all $x\in \mathbb R^{\times}$. 
Let $n\in \mathbb Z$. Observe that $\nu^n \times \omega$, when restricted to $\SL_2(\mathbb R)$, is a sum of two (limits of) 
 discrete series representations with the lowest $\SO_2$-types $\pm(|n|+1)$. 

\smallskip 
\underline{Case $E=\mathbb R^3$ and $K=\mathbb R^2$.} Let $\chi=(\chi_1,\chi_2,\chi_3)$ be a unitary character of $(\mathbb R^{\times })^3/\Delta(\mathbb R^{\times})$.  
The packet $P(E,K,\chi)$ consists of representations appearing in the restriction to $\GL_2(\mathbb R^3)^{\det}$ of
\[
(\chi_1\times 1) \otimes (\chi_2\times 1)\otimes (\chi_3\times 1). 
\]
 This representation is irreducible when restricted to $\SL_2(\mathbb R^3)$ unless $\chi_i=\omega$ for at least one $i$. The group 
$\GL_2(\mathbb R^3)^{\det}$ is large enough so that the restriction is still irreducible if precisely one $\chi_i$ is $\omega$. In view of the relation 
$\chi_1\cdot \chi_2\cdot \chi_3=1$, at most two $\chi_i$ can be $\omega$, and this is precisely when $\chi$ is a non-trivial quadratic character. Then and only then 
the packet consists of two elements. The standard intertwining operator provides an identification of $P(E,K,\chi)$ and $P(E,K,\chi^{-1})$. 

\smallskip 
\underline{Case $E=\mathbb R^3$ and $K=\mathbb C$.}
Let $\chi=(n_1,n_2,n_3)$ be a character of $(\mathbb T)^3/\Delta(\mathbb T)$.  
The packet $P(E,K,\chi)$ consists of representations appearing in the restriction to $\GL_2(\mathbb R^3)^{\det}$ of
\[
(\nu^{n_1}\times \omega) \otimes (\nu^{n_2}\times \omega)\otimes (\nu^{n_3} \times \omega). 
\]
The restriction to $\SL_2(\mathbb R^3)$ consists of 8 summands, hence the packet $P(E,K,\chi)$ consists of 4 representations.

\smallskip 
\underline{Case $E=\mathbb R\times \mathbb C$ and $K=\mathbb R^2$.}
 The restriction from $\GL_2(\mathbb R\times \mathbb C)$ to $\GL_2(\mathbb R\times \mathbb C)^{\det}$ is always irreducible, hence the packets are singletons. 
Let $\chi=(\chi_{\mathbb R}, \chi_{\mathbb C})$ be a unitary character of $(\mathbb R^{\times}\times \mathbb C^{\times} )/\Delta(\mathbb R^{\times})$.  
The packet $P(E,K,\chi)$ consists of the restriction to $\GL_2(\mathbb R \times \mathbb C)^{\det}$ of 
\[
(\chi_{\mathbb R}\times 1) \otimes (\chi_{\mathbb C}\times 1). 
\]

\smallskip 
\underline{Case $E=\mathbb R\times \mathbb C$ and $K=\mathbb C$.} 
We are again restricting from $\GL_2(\mathbb R\times \mathbb C)$ to $\GL_2(\mathbb R\times \mathbb C)^{\det}$ hence the packets are singletons. 
Let $\chi=(m, \chi_{\mathbb C})$ be a unitary character of $(\mathbb T\times \mathbb C^{\times} )/\Delta(\mathbb T)$.  
The packet $P(E,K,\chi)$ consists of the restriction to $\GL_2(\mathbb R \times \mathbb C)^{\det}$ of 
\[
(\nu^m \times \omega) \otimes (\chi_{\mathbb C}\times 1). 
\]

\smallskip 
Summarizing, we have 4 families of tempered packets 
$P(E,K,\chi)=P(E,K,\chi^{-1})$ of $\GL_2(E)^{\det}$, parameterized by unitary characters $\chi$ of $T_{E,K}$. If $E=\mathbb R^3$ and $K=\mathbb C$, then 
$|P(E,K,\chi)|=4$. As a part of our main result, the 4 members of this packet are naturally parameterized by 4 embeddings $\mathbb R^3 \rightarrow H_3(\mathbb C)$. 
If $\chi$ is a non-trivial quadratic character (this happens only  if $E=\mathbb R^3$ and $K=\mathbb R^2$) then $|P(E,K,\chi)|=2$. Let $\pi^{+}(\chi), \pi^{-}(\chi)$ be its constituents. 
Otherwise $|P(E,K,\chi)|=1$ and its unique element will be denoted by $\pi(\chi)$. 

\subsection{Main result}  Let $V$ be the Harish-Chandra module of the minimal representation of $H_J$. Consider the dual pair 
$G_E \times T_{E,K}$ corresponding to an embedding $E\rightarrow J$.  Let $K_E$ be a maximal compact subgroup of $G_E$. 
Let $\rho$ be an irreducible representation of $T_{E,K}$. 
Then there exists a $(\mathfrak g_E, K_E)$-module $\Theta(\rho)$, such that 
\[ 
\Theta(\rho)\otimes \rho \cong V /\cap_{\varphi \in \Hom(V, \rho)} \ker(\varphi) 
\] 
where $\varphi$ are homomorphisms in the sense of Harish-Chandra modules. 
Observe that a priori $\Theta(\rho)$ could be trivial or infinite length otherwise. 
The following is the main result of this paper, that will be proved in the next two section. 

\smallskip

\begin{thm} \label{T:main} 
 Let $G_E\times T_{E,K}$ be the dual pair arising from an embedding $E\rightarrow J$.
Let $\chi$ be a unitary character of $T_{E,K}^{\circ}$. 
\begin{itemize} 
\item  If $E\rightarrow J$ is not one of the 4 embeddings $\mathbb R^3 \rightarrow H_3(\mathbb C)$, then 
 $\Theta(\rho(\chi))\cong L(\pi(\chi), 1)$, unless  $\chi$ is quadratic and non-trivial, and then 
$\Theta(\rho^{\pm}(\chi))\cong L(\pi^{\pm}(\chi),1)$. 
\item If $E\rightarrow J$ is one of the 4 embeddings $\mathbb R^3 \rightarrow H_3(\mathbb C)$, then 
 $\Theta(\rho(\chi))\cong L(\pi, 1)$, where $\pi\in P(E,K,\chi)$. As we run through all 4 embeddings $\mathbb R^3 \rightarrow H_3(\mathbb C)$, 
 $\pi$ runs through the 4 representations in $P(E,K,\chi)$.
 \end{itemize} 

\end{thm} 

The representation $\Theta(\epsilon)$ is always irreducible, and can be described as it sits in certain degenerate principal series representations, along with 
$\Theta(\rho(1))$. 
Let $I_E(s)$ denote the (normalized) degenerate principal series for $G_E$ where we induce $|\det|^s$ from $P_E$. Let $I^{\omega}_E(s)$ be the quadratic twist 
of this series, i.e. we induce $\omega(\det) \cdot |\det|^s$. (Recall that $\omega$ is the sign character of $\mathbb R^{\times}$.) 
The following result is due to Avner Segal  \cite[Appendix A]{Se} but formulated with our interpretation in terms of theta lifts.

\begin{thm} \label{T:main_2} 
 Let $\Theta_{E\rightarrow J}(\rho)$ denote the theta lift of $\rho$ in the correspondence arising from the embedding $E\rightarrow J$. 

\begin{itemize} 
\item For every $E$, we have an exact sequence 
\[ 
0 \rightarrow \oplus \Theta_{E\rightarrow H_3(\mathbb C)} (\epsilon) \rightarrow I_E(1/2) \rightarrow \Theta_{E\rightarrow M_3(\mathbb R)} (\rho(1)) \rightarrow 0 .
\] 
\item For every $E$, we have an exact sequence 
\[ 
0 \rightarrow  \Theta_{E\rightarrow M_3(\mathbb R)} (\epsilon) \rightarrow I^{\omega}_E(1/2) \rightarrow \oplus \Theta_{E\rightarrow H_3(\mathbb C)} (\rho(1)) \rightarrow 0 .
\] 
\end{itemize} 
Here $H_3(\mathbb C)$ is any Jordan algebra arising from an involution of the second kind on $M_3(\mathbb C)$.  The sum, in both sequences, is over the isomorphism 
classes of embeddings of $E$ into such algebras. There is one class if $E=\mathbb R\times \mathbb C$, and four if $E=\mathbb R^3$. 
 
  \end{thm} 

\noindent
{\bf Remark:} $M_E \times T_{E,K}$ is a dual pair in a Levi subgroup $M_J$ of $H_J$ of absolute type $A_5$. The correspondence obtained by 
restricting the minimal representation of $M_J$ to 
this dual pair is essentially a triple tensor product of the classical theta correspondence for the dual pair $\SL_2 \times \O_2$, 
and it works out to be $\pi(\chi) \leftrightarrow \rho(\chi)$ \cite[Section 10]{GS22}. Since the minimal representation of $M_J$ is a quotient of the 
minimal representation of $H_J$, it follows at once that $\Theta(\rho(\chi))$ contains $L(\pi(\chi), 1)$ as a subquotient. This gives a lower bound on $\Theta(\rho(\chi))$. 
The upper bound is achieved by working out the types of $\Theta(\rho(\chi))$. This strategy, however, fails for $\rho=\epsilon$ as it does lift to a representation of 
$M_E$.

\section{Correspondence of infinitesimal characters} 
We shall describe a compact form of the $D_4$-dual pair in the real group $E_{6,2}$, that 
will be used to determine the correspondence of infinitesimal characters. In this section Jordan algebras play no role, and notation will be somewhat different.

 Let $\mathfrak h$ be the split Lie algebra of type $E_6$,  and 
we fix a Chevalley basis of $\mathfrak h$, over a general field $F$.  
We realize the root system $E_6$ siting in $\mathbb R^8$, in the standard fashion, where the simple roots are  
 \[ 
 \alpha_1=\frac12(e_1+e_8)-\frac12(e_2+e_3+e_4+e_5+e_6+e_7) 
 \] 
 \[ 
 \alpha_2=e_1+e_2, ~\alpha_3=e_2-e_1, ~\alpha_4=e_3-e_2, ~\alpha_5=e_4-e_3, ~\alpha_6=e_5-e_4. 
 \] 
 We shall realize the co-character lattice in the same space, so that the natural pairing between co-characters and roots is given by 
 the standard dot product on $\mathbb R^8$ denoted by $\langle \cdot, \cdot \rangle$. 
Let $\alpha_0$ be the lowest root. The lowest root and the simple roots give rise to an extended Dynkin diagram that has a symmetry of order 3. 
Let $\lambda_1$ be a miniscule fundamental co-character such that $\langle \lambda_1, \alpha_1 \rangle=1$, 
$\langle \lambda_1, \alpha_i \rangle=0$ for other simple roots, and $\langle \lambda_1, \alpha_0 \rangle=-1$, as pictured. 

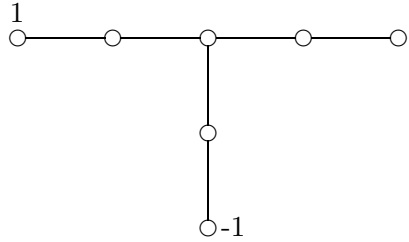
\begin{picture}(200,120)(-120,-15)

\put(79,73){\line(0,-1){30}}
\put(79,40){\circle{6}}
\put(79,37){\line(0,-1){30}}
\put(79,04){\circle{6}}

\put(84,01){-1}

\put(04,82){1}

\put(07,76){\circle{6}}
\put(10,76){\line(1,0){30}}

\put(43,76){\circle{6}}
\put(46,76){\line(1,0){30}}
\put(79,76){\circle{6}}

\put(82,76){\line(1,0){30}}
\put(115,76){\circle{6}}

\put(118,76){\line(1,0){30}}
\put(151,76){\circle{6}}

\end{picture}

The group $S_3$ of automorphisms of the extended diagram lifts to the group of automorphisms of the root system and to $H$, the adjoint group of 
type $E_6$. 
 Let $\sigma$ be the automorphism of order 3 that corresponds to the counterclockwise rotation of the 
diagram by $120$ degrees, so to say.  Let $\lambda_2=\sigma(\lambda_1)$ and $\lambda_3=\sigma(\lambda_2)$. Then $-\lambda_2$ is the other 
miniscule fundamental weight and $\lambda_1 + \lambda_2 +\lambda_3=0$, in the additive notation. In the multiplicative notation $\lambda_i$ are 
homomorphisms from $\mathbb G_m$ into the adjoint Chevalley group and 
$\lambda_1(t)\lambda_2(t)\lambda_3(t)=1$ for all $t\in F^{\times}$. 
Thus $\lambda_i(t)$ generate a torus $T$ isomorphic to $(F^3)^{\times}/F^{\times}$. 
Centralizing $T$, there is a split, simply connected group $G$ of type $D_4$ corresponding to the roots perpendicular to the three $\lambda_i$. 
The group $S_3$ simultaneously acts on $T$ and $G$.  

\smallskip Assume now $F=\mathbb R$. Now $\tau=\lambda_1(-1)$ defines a Cartan involution of $\mathfrak g$ and the corresponding real form  $H'$ of $H$ is denoted 
by $E_{6,2}$ in the classification of real groups. 
 Since $\tau$ commutes with $T\times G$ we have a compact form $T'\times G'$ of the dual pair in $H'$. 
Each $\lambda_i$ defines a homomorphism from $\U(1)$ into $H'$ and these combine to give an isomorphism  $T'\cong \U(1)^3/\U(1)$. 
For every integer $n$ let $\chi_n(z)=z^n$ be the character of $\U(1)$. 
Any character of $T'$ is 
\[ 
\chi(a,b,c)=\chi_a\boxtimes \chi_b \boxtimes \chi_c
\] 
 for three integers such that $a+b+c=0$. 
Let $K'$ be the maximal compact subgroup of $H'$ picked by $\tau$. It is 
isomorphic to $\U(1) \times \mu_4 ~\Spin_{10}$ where the map from $\U(1)$ into $K'$ is given by $\lambda_1$, while $\Spin_{10}$ corresponds to 
simple roots $\alpha_i$ for all $i\neq 1$. 
These roots are linear combinations of $e_1, \ldots, e_5$ and are a part of the standard realization of $D_5$ root system.  
Then 
\[ 
\omega=\frac12 (e_1 + e_2 + e_3 + e_4 + e_5)
\] 
 is the highest weight of a half-spin representation $V_{\omega}$ of $\Spin_{10}$.  The $K'$-types of the minimal representation $V$ of $H'$ are \cite{Ze} 
\[ 
V =\oplus_{n=0}^{\infty}~ \chi_{n+4} \boxtimes V_{n\omega}. 
\] 
The group $G'\cong \Spin_8$ corresponds to roots $\alpha_i$ for $i\neq 1,6$. These roots are linear combinations of $e_1, \ldots, e_4$. Hence 
centralizing $\Spin_8$ in $\Spin_{10}$ is $\U(1)$, the image of the co-character $2e_5$.  We have a restriction formula \cite{Li}, from $\Spin_{10}$ to $\Spin_8\times \U(1)$
\[ 
V_{n\omega}= \bigoplus  V_{(\frac{n}{2}, \frac{n}{2}, \frac{n}{2}, \frac{b}{2})}  \boxtimes  \chi_{b} 
\] 
where the sum is over all integers $b$ such that $|b|\leq n$ and $n\equiv b \pmod{2}$ and $(\frac{n}{2}, \frac{n}{2}, \frac{n}{2}, \frac{b}{2})$ is the highest 
weight for $\Spin_8$ in the standard realization of the $D_4$ root system. 
Substituting this in the formula for $V$, and using $-2e_5=\lambda_1+2\lambda_2$,  the restriction of $V$ to $T'\times G'$ is a 
multiplicity free sum of the terms 
\[ 
\chi(n+4, -\frac{b+n}{2} -2, \frac{b-n}{2} -2) \boxtimes V_{(\frac{n}{2}, \frac{n}{2}, \frac{n}{2}, \frac{b}{2})}. 
\] 

\smallskip 
Let $\mathfrak h_{\mathbb C}=\mathfrak h\otimes \mathbb C$ etc. 
The annihilator of $V$ in the enveloping algebra $U(\mathfrak h_{\mathbb C})$ is the Joseph ideal $I$.  Let 
$Z(\mathfrak g_{\mathbb C})$ be the center of the enveloping algebra $U(\mathfrak g_{\mathbb C})$. 
The decomposition of $V$,  under the restriction to $T'\times G'$, shows that there is a map 
 $\varphi : Z(\mathfrak g_{\mathbb C}) \rightarrow U(\mathfrak t_{\mathbb C})$
 such that  $z\equiv \varphi(z)\pmod I$.    More precisely, we have the following lemma: 
 
 \begin{lemma} \label{L:infinitesimal}  
  Let $(a,b,c)$ be an infinitesimal character of $\mathfrak t_{\mathbb C}$,  that is,  a triple of complex numbers such that $a+b+c=0$. Then 
 the pullback of $(a,b,c)$ by $\varphi$ is an infinitesimal character of $\mathfrak g_{\mathbb C}$ represented by 
 \[ 
\frac{1}{2}(a+2,-a+2,b+c,b-c). 
\] 
in the standard realization of the $D_4$ root system. 
\end{lemma} 
\begin{proof} This is an easy check using the decomposition of $V$ and that the infinitesimal character of $V_{(\frac{n}{2}, \frac{n}{2}, \frac{n}{2}, \frac{b}{2})}$  is
$(\frac{n}{2}+3, \frac{n}{2}+2, \frac{n}{2}+1, \frac{b}{2})$ 
\end{proof}

Of course, Lemma \ref{L:infinitesimal} gives a correspondence of infinitesimal characters for any real form of the dual pair, since the annihilator of the minimal 
representation is always the Joseph ideal \cite{GS04}, and the dual pair is unique up to conjugation over complex numbers. 
It is more symmetric to write the correspondence in the above lemma as 
\[
 \beta + \frac{a}{2} \alpha_1 + \frac{b}{2} \alpha_2 + \frac{c}{2} \alpha_3 \longleftrightarrow
  (a,b,c)
\]
where $a+b+c =0$, $\beta$ the highest root, and $\alpha_i$ are three simple roots permuted by $S_3$. 

\section{The correspondences in the case $J=M_3(\mathbb R)$}

\subsection{Minimal representation} 
Let $K_J$ be a maximal compact subgroup of $H_J$. Recall that in the case $J=M_3(\mathbb R)$ the group $H_J$ is split. Hence 
$K_J$ is isomorphic to $\Sp(4)/\mu_2 \times \mathbb Z/2\mathbb Z$.  The non-trivial element 
of $\mathbb Z/2\mathbb Z$ is the Cartan involution,  giving a decomposition 
\[ 
\mathfrak h_J =\mathfrak k_J \oplus \mathfrak p_J
\] 
where $\mathfrak p_J\otimes \mathbb C  \cong V_{\omega_4}$, the fourth fundamental representation of $\Sp(4)$. 
Let $V$ be the $(\mathfrak h_J, K_J)$-module corresponding to the minimal representation of $H_J$ \cite{BK}. We have an isomorphism of $\Sp(4)$-modules,
\[ 
V= \oplus _{n=0}^{\infty} V_{n\omega_4}, 
\] 
where $V_{n\omega_4}$ is the irreducible representation of $\Sp(4)$ of the highest weight $n\omega_4$. 
The nontrivial element in $\mathbb Z/2\mathbb Z\subset K_J$, i.e. the Cartan involution, acts on $V_{n\omega_4}$ by $(-1)^n$.  

\subsection{Some branching rules} In this paper the highest weight $\mu$ of 
a finite-dimensional representation of $\Sp(2)$ is given by a pair of integers $\mu=(x,y)$ such that $x\geq y \geq 0$.  In this setup $(1,0)$ and $(1,1)$ are fundamental 
highest weights corresponding to 4 and 5 dimensional representations, respectively. 
Let $V_{\mu}$ be the irreducible representation of 
$\Sp(2)$ with the highest weight $\mu$. We shall need several branching rules for small rank groups. 
They are either special cases of branching rules in the literature, see \cite{HTW} 
and \cite{GoW}, or can be directly derived from the Weyl character formula. 

\begin{lemma}\label{L:branching_1} For the natural inclusion of $\Sp(2) \times \Sp(2)$ into $\Sp(4)$, 
 the restriction of  $V_{n\omega_4}$ to $\Sp(2) \times \Sp(2)$ is isomorphic to a multiplicity free sum of 
$V_{\mu} \boxtimes V_{\mu}$ over all highest weights $\mu=(x,y)$ such that $n\geq x$. 
\end{lemma} 

Let $V_n$ be the irreducible $n+1$-dimensional representation of $\SU_2\cong \Sp(1)$. 

\begin{lemma}\label{L:branching_2} Let $\mu=(x,y)$ be a highest weight for $\Sp(2)$. 
The restriction of $V_{\mu}$ to $\SU_2 \times \SU_2$ is a multiplicity free sum of $V_a \boxtimes V_b$ over all $(a,b)$ such that $a+b\equiv x+y \pmod{2}$ 
\[ 
|a-b| \leq x-y \text{ and } x-y\leq a+b \leq x+y. 
\] 
\end{lemma} 

Finally, we need a branching rule from $\Sp(2)$ to $\Sp(1) \times \SO_2$, where $\Sp(1) \times \O_2 \subset \Sp(2)$ is a classical dual pair. We
 work in the equivalent setting of restriction from $\SO(5)$ to $\SO(3) \times \SO(2)$.  In this setup the highest weights of irreducible representations of $\SO(5)$ are 
 pairs $a\geq b \geq 0$ of half integers such that $a\equiv b \pmod{\mathbb Z}$. The corresponding $\Sp(2)$-highest weights are $(x,y)=(a+b,a-b)$. 
  The highest weights of irreducible representations of $\SO(3)$ are non-negative half integers $c$ and 
 representations of $\SO(2)$ are characters $\chi_n$ where $n$ is a half integer. 
  For every non-negative half integer $n$ let 
  \[ 
  A(n) = \chi_{n} + \chi_{n-1} + \cdots + \chi_{-n}. 
  \] 
  For every non-negative integer $n$ let 
   \[ 
  B(n) = \chi_{n} + \chi_{n-2} + \cdots + \chi_{-n}. 
  \] 
  These are $2n+1$ and $n+1$-dimensional representations of $\SO(2)$, respectively. We have the following branching result: 
  
  \begin{lemma} \label{L:branching_3} Let $\mu=(a,b)$ be a highest weight for $\SO(5)$. The restriction of $V_{\mu}$ to
  $\SO(3) \times \SO(2)$  is 
  \[ 
  V_{\mu} = \oplus_{c} V_c\boxtimes \chi[c] 
  \] 
  where the sum is over all irreducible representations $V_c$ of $\SO(3)$. Here $\chi[c]$ is zero unless $a\geq c$ and $a \equiv c  \pmod{\mathbb Z}$, and then 
  \begin{itemize} 
  \item $\chi[c]= A(b) B(a-c ) $ if $a\geq c\geq b$ and 
  \item 
  $\chi[c]= A(c) B(a-b)$ if $a\geq b\geq c$. 
  \end{itemize} 
  
  \end{lemma}

\subsection{Split case}

Assume that $E=\mathbb R^3$. In this case $T^{\circ}_{E,K}\cong  (\mathbb R^{3})^{\times} /\mathbb R^{\times}$ and a character 
$\chi$ is a triple $(\chi_1,\chi_2, \chi_3)$ of characters of $\mathbb R^{\times}$ such that $\chi_1 \cdot \chi_2 \cdot \chi_3=1$. Let 
$\nu_i\in \mathbb C$ be the differential of $\chi_i$. Then $(\nu_1,\nu_2,\nu_3)$, $\nu_1 + \nu_2 + \nu_3=0$, 
is the infinitesimal character of $\chi$ and the lift $\Theta(\chi)$ has the infinitesimal character 
\[
 \beta + \frac{\nu_1}{2} \alpha_1 + \frac{\nu_2}{2} \alpha_2 + \frac{\nu_3}{2} \alpha_3. 
\]
 The character $\chi$ is unitary if and only if all three $\nu_i$ are purely imaginary,

\smallskip 
 The group $G=G_E$ is split. 
Fix a maximal compact subgroup $K$ of $G$. It is isomorphic  $\SU_2^4 /\mu_2$. Fix such an isomorphism. Then
any irreducible representation of $K$ is isomorphic to  $V(a,b,c,d)=V_a\boxtimes V_b \boxtimes V_c \boxtimes V_d$, such that $a+b+c+d$ is even,  where $V_n$ denotes  the irreducible 
representation of $\SU_2$ of dimension $n+1$.   The center  $Z_K$ of $K$ is $\mu_2^4/\mu_2$. The Lie algebra $\mathfrak g$  of $G$ decomposes under the action of $K$ as 
\[ 
\mathfrak g =\mathfrak k \oplus \mathfrak p
\] 
where $\mathfrak p\otimes \mathbb C \cong V(1,1,1,1)$. The center $Z_K$ acts on $\mathfrak p$ by a character $\chi_K$. It is evident that the center $Z_G$ of $G$ is the kernel of $\chi_K$.

Let $K_J \supseteq K$ be a maximal compact subgroup of $H_J$, and fix an isomorphism 
 $K_J\cong \Sp(4)/\mu_2 \times \mathbb Z/2\mathbb Z$.  Up to $\Sp(4)$-conjugation, there is only one way to embed
$\SU_2^4$ into $\Sp(4)$. 
Under the action of $K\cong \SU_2^4/\mu_2$ 
 \[ 
 \mathfrak p_J\otimes \mathbb C \cong V_{\omega_4} = V(1,1,1,1) \oplus 2 V(0,0,0,0) \oplus_{(a,b,c,d)\in S} V(a,b,c,d) 
 \] 
 where $S$ contains 6 elements, all permutations of $(1,1,0,0)$. The center $Z_G$ acts trivially only on the first two summands. It now easily follows that the centralizer of 
 $Z_G$ in $\mathfrak h_J$ is $\mathfrak g \oplus \mathfrak a$ where $\mathfrak a\otimes \mathbb C \cong 2V(0,0,0,0)$. Since $[\mathfrak a, \mathfrak a]\subseteq \mathfrak k_J$, and 
 commutes with $K$, it follows that $[\mathfrak a, \mathfrak a]=0$. Let 
 $A=\exp(\mathfrak a)\subseteq H_J$, where the exponential is taken in the sense of Lie groups. Let $K'$ be a subgroup of $K_J$, isomorphic to $Z_K$, via the isomorphism 
$x \mapsto (x, \chi_K(x)) \in \Sp(4)/\mu_2 \times \mathbb Z/2\mathbb Z$.  Note that $Z_K \cap K'=Z_G$ and, by design, $K'$ centralizes $\mathfrak g$.  The conjugation 
action of $K'$ on $A$ descends to an action of $K'/Z_G$ where the non-trivial element acts by the inverse on $A$. The group $K'A$ is  our $T_{E,K}$, and 
$Z_G A$ is $T^{\circ}_{E,K} \cong (\mathbb R^{3})^{\times} /\mathbb R^{\times}$. 
This completes a description of the dual pair $G \times T_{E,K}$ suitable for studying the restriction problem for real groups.

 \subsection{Representations of split $G$} 

From Lemma \ref{L:infinitesimal} representations of $G$ that appear as theta lifts have the infinitesimal character $\lambda =(s+1,s,s-1,u)$  in the usual Bourbaki notation for the 
root system of the type $D_4$. Observe that $s$ and $u$ are purely imaginary when lifting unitary representations of split $T_{E,K}$.

\vskip 10pt 
We need the following general result, where $G$ is any split group.  corresponding to a root system $\Phi$. 
 Let $P_0=M_0 A_0 N_0$ be a minimal parabolic. Let $\Phi$ be the root system of $G$. 
 Fix a complex valued linear functional $\lambda$ of $\mathfrak a_0$, the Lie algebra of $A_0$. Let $\Phi_{\lambda}$ be the set of roots $\alpha$ such that 
 $(\lambda, \alpha^{\vee}) \in \mathbb R$. Write $\lambda = \Re(\lambda) + i \Im(\lambda)$, the sum of real and imaginary parts. Observe that 
 $\Phi_{\lambda}$ is simply the set of roots such that $\alpha^{\vee}$ is perpendicular to $\Im(\lambda)$. Hence $\Phi_{\lambda}$ is the set of 
 roots in a Levi subgroup. Without loss of generality, we shall assume that $\Im(\lambda)$ is contained in the closure of the positive chamber determined by $P_0$. 
 Then $\Phi_{\lambda}$ is the root system of a Levi of a parabolic $P=MAN$ containing $P_0$.  In this setting we have the following, Lemma 4 in \cite{Vo}: 
 
 \begin{lemma}\label{L:vogan} 
 Fix a character $\delta$ of $M_0$. 
Irreducible subquotients  of $\Ind_{P_0}^G(\delta,e^{ \lambda})$ are induced from irreducible subquotients of $\Ind_{P_0}^P(\delta, e^{\lambda})$.
 \end{lemma}

\smallskip 
At this point it may be useful to explain the strategy of computing the theta correspondence in this case. 
By analogy with the $p$-adic case, representations of $G$ appearing in the correspondence should be 
constituents of the degenerate principal series, obtained by inducing one-dimensional representations of a parabolic group 
$P=MA_PN$ of $G$ such that the connected component of $M$ is a simply connected group corresponding to an $A_2$ root subsystem. 
In particular, we can identify this group with $\SL_3(\mathbb R)$. It is not difficult to check that $M=\SL_3(\mathbb R) \cdot Z_G$, where 
$Z_G$ is the center of $G$. The group $A_P$ is 2-dimensional. There exists an isomorphism $i: Z_G A_P \rightarrow Z_G A$ such that the following holds. 

\begin{prop} \label{P:degenerate} 
Let $\chi$ be a unitary character of $Z_GA$ and let $\chi_P$ be the character of $Z_G A_P$ obtained by pulling back $\chi$ by $i$. 
 Extend $\chi_P$ to a character of $MA_P$ trivial on $\SL_3(\mathbb R)$. 
Let $\Ind_P^G(\chi_P)$ be the unitarizable degenerate principal series representation obtained by inducing $\chi_P$. 
\begin{itemize} 
\item If $\chi$ is not quadratic, then $\Ind_P^G(\chi_P) \cong L(\pi(\chi), 1)$. 
\item If $\chi$ is quadratic but not trivial, then $\Ind_P^G(\chi_P) \cong L(\pi(\chi)^+, 1)\oplus L(\pi(\chi)^-, 1)$.
\item If $\chi=1$, then $\Ind_P^G(1)\cong L(\rho(1),1) \oplus \Sigma$ where $\Sigma$ is irreducible with the minimal $K$-type $V(1,1,1,1)$. 
\end{itemize} 
\end{prop} 
\begin{proof} By a classical result of Bruhat \cite{B}, $\Ind_P^G(\chi_P)$ is irreducible unless $\chi_P$ is quadratic in which case it can have at most two summands.  
 It is straightforward to check that  the submodules of $\Ind_P^G(\chi_P)$ generated by minimal types are claimed Langlands quotients. This completes the first two 
 bullets. For the last, using the {\tt atlas} software, one checks that the complement of $ L(\rho(1),1)$ in $\Ind_P^G(1)$ has the minimal type $V(1,1,1,1)$. But there is one 
 representation at the given infinitesimal character with this minimal type, see Section \ref{SS:split}.  
\end{proof} 

The value of Proposition \ref{P:degenerate} is that it gives us a control of $K$-types of $L(\pi(\chi),1)$. Thus proving that $\Theta(\rho(\chi))=L(\pi(\chi),1)$ 
will be accomplished by controlling $K$-types of $\Theta(\rho(\chi))$.  More precisely, 
let $K_M\cong \SO_3 \cdot Z_G$ be a maximal compact subgroup of $M$. If we pick $K\supset K_M$, then 
 the factor $\SO_3\cong \SU_2/\mu_2$ embeds diagonally into $K\cong \SU_2^4/\mu_2$. 
 Thus all $K$-types of representations $\Ind_P^G(\chi_P)$ have non-zero 
 $\SO_3$-invariant vectors. Moreover, the multiplicity of the type is equal to the dimension of the space of $\SO_3$-invariants, if it appears. (If it appears or not, this is decided by 
 the central character.)  Thus one should expect that the
 $K$-types appearing in the restriction of $V$ have non-zero $\SO_3$-fixed vectors. Indeed, this follows from Lemma \ref{L:branching_1} and Lemma \ref{L:branching_2}. A
 type $V(a,b,c,d)$ can have non-zero $\SO_3$-fixed vectors only if 
 \[ 
 a + b + c \geq d 
 \] 
 holds, as well as three other inequalities obtained by cyclically permuting $a,b,c$ and $d$. 
 We record the following more precise information that we shall need later. 

\begin{prop} \label{P:branching} 
Let $V(a,b,c,0)$ be a $K$-type containing a line, necessarily unique, invariant under $\SO_3$. Then the multiplicity of 
$V(a,b,c,0)$ in $V_{n\omega_4}$ is 
\[ 
n+1 -\frac{a +b + c}{2} 
\] 
whenever this integer is not negative, and it is 0 otherwise. If $V(a,b,c,0)$ does not contain a $\SO_3$-invariant line,
 then the multiplicity of $V(a,b,c,0)$ in $V_{n\omega_4}$ is $0$. 
\end{prop} 

Observe that an $\SO_3$-invariant line exists precisely when $a$, $b$ and $c$ are sides of a triangle. 
Thus, if this condition is met, $V(a,b,c,0)$ appears the first time for $n=(a+b+c)/2$, with multiplicity 1, and with multiplicities $2,3,4, \ldots$, thereafter. Of course, 
instead of assuming  $d=0$, we can assume that any other coordinate is 0, and analogous statements hold. 

\smallskip 

Recall that $V= \oplus _{n=0}^{\infty} V_{n\omega_4}$ is the $(\mathfrak h_J, K_J)$-module corresponding to the minimal representation of $H_J$. 
  Let $V_n=\oplus _{i\leq n} V_{i\omega_4}$. 
Let $U_n(\mathfrak h_J)$ be the usual filtration of the universal enveloping algebra $U(\mathfrak h_J)$ by the degree. 
By examining the weights of  
\[ 
\mathfrak p_J \otimes V_{i\omega_4}\cong V_{\omega_4} \otimes V_{i\omega_4},
\] 
 it is clear that it cannot contain $V_{j\omega_4}$ as a summand if $j> i+1$.  Thus, 
\[ 
U_n(\mathfrak h_J) \cdot V_m \subseteq V_{n+m} 
\] 
for all $n$ and $m$.  We now need the following key lemma.

\begin{lemma} \label{L:key} 
 $V$ is a torsion free $U(\mathfrak a)$-module. 
\end{lemma} 
\begin{proof}  Let $v\in V$ and consider the matrix coefficient $f_v(a)=(\pi(\exp a)v, v)$. By a theorem of Howe and Moore, since $V$ is unitary, $f_v$ decays 
 in all directions on $\mathfrak a$. 
By asymptotic expansion of matrix coefficients \cite{CM}, this decay is exponential. 
Thus the Fourier transform $\hat{f}_v$ is a holomorphic function in a tubular neighborhood of $i\mathfrak a^*$ in $\mathfrak a^*_{\mathbb C}$. 
The enveloping algebra $U(\mathfrak a)$ is naturally 
the algebra of polynomials on $\mathfrak a_{\mathbb C}$. If $p\in U(\mathfrak a)$ annihilates $v$, hence also $f_v$, then $p\cdot \hat{f}_v =0$ on the tubular neighborhood. 
This implies $p=0$. 

\end{proof} 

\begin{prop} \label{P:free} 
Let $V(a,b,c,0)$ be a $K$-type containing an $\SO_3$-invariant line. Then 
\[ 
\Hom_K(V(a,b,c,0), V) 
\] 
is a free $U(\mathfrak a)$-module generated by a non-zero element in $\Hom_K(V(a,b,c,0), V_{\frac{a+b+c}{2}\omega_4})$. 
\end{prop} 
\begin{proof} Assume, for simplicity, that the $K$-type is trivial. We need to prove that $V^K$ is a free $U(\mathfrak a)$-module  generated by 
a $K_J$-spherical vector $v_0$. Since the action of $U(\mathfrak a)$ is torsion-free, the map $x\mapsto x\cdot v_0$ is an injection of $U_n(\mathfrak a)$ into 
$V_n^K$. By Proposition \ref{P:branching}, the dimension of $U_n(\mathfrak a)$ is equal to the dimension of $V_n^K$. 
Hence $V^K = U(\mathfrak a) \cdot v_0 \cong U(\mathfrak a)$. The general case is treated analogously. 
\end{proof}

Let $\chi$ be a character of the split torus $Z_G A$. If $\chi^2\neq 1$, then $\chi$ induces to an irreducible 2-dimensional representation $\rho(\chi)$ of $K'A$, such 
that $\rho(\chi)\cong \rho(\chi)^{-1}$. If $\chi^2=1$, then $\chi$ extends to two characters $\rho(\chi)^{\pm}$ of $K'A$. These are undistinguishable except when $\chi$ is 
trivial and then the two extensions will be denoted by $\rho(1)$ and $\epsilon$.  The character $\chi$ is determined by its restriction $\chi_G$ to $Z_G$ and its differential 
$\nu=d\chi : \mathfrak a \rightarrow \mathbb C$. Abusing notation, we shall use $\nu$ to denote the corresponding homomorphism of $U(\mathfrak a)$ and 
$\mathbb C_{\nu}$ the one-dimensional representation of $U(\mathfrak a)$ on which $U(\mathfrak a)$ acts by $\nu$.  Let 
\[ 
\Theta(\nu) =V / \cap_{\varphi\in \Hom_{U(\mathfrak a)}(V, \mathbb C_{\nu})} ker(\varphi). 
\] 
The following is an immediate consequence of of Proposition \ref{P:free} 

\begin{cor}  \label{C:m_one} 
Let $V(a,b,c,0)$ be a $K$-module containing an $\SO_3$-invariant line.  For every $\nu$, $\Theta(\nu)$ contains the type $V(a,b,c,0)$ with multiplicity one. 
Moreover, the type is contained in the projection of $V_{\frac{a+b+c}{2}\omega_4}$. 
\end{cor} 

Each $\Theta(\nu)$ can be decomposed under the action of $Z_G$, 
\[ 
\Theta(\nu) =\oplus_{d\chi=\nu} \Theta(\chi).
\] 
Since $Z_G$ is explicitly given as a subgroup of $Z_K$ Corollary \ref{C:m_one} can be easily refined to a statement about $\Theta(\chi)$. For example, 
if the restriction of $\chi$ to $Z_G$ is trivial, then $\Theta(\chi)$ contains only the types $V(a,b,c,d)$ such that  $a,b,c$ and $d$ have the same parity. 
If $\nu\neq 0$ then $\Theta(\chi)=\Theta(\rho(\chi))$. On the  other hand, all $\chi$ such that $d\chi=0$ are quadratic, and 
$\Theta(\chi)=\Theta(\rho(\chi)^+) \oplus \Theta(\rho(\chi)^-).$
Since the action of $K'$ on $V(a,b,c,0)\subseteq V_{\frac{a+b+c}{2}\omega_4}$ is known, we can refine 
Corollary \ref{C:m_one} further to a statement about $\Theta(\rho(\chi)^{\pm})$. For example, we have the following: 

\begin{cor} \label{C:m_two} Let $a,b,c$ three even integers, sides of a triangle. Then $V(a,b,c,0)$ appears  in $\Theta(\rho(1))$ if $(a+b+c)/2$ is even and 
in $\Theta(\epsilon)$ if $(a+b+c)/2$ is odd. 
\end{cor} 

Summarizing here is what we know about $\Theta(\chi)$, the third bullet a consequence of Lemma \ref{L:infinitesimal}: 

\begin{itemize} 
\item $K$-types are only those that contain an $\SO_3$-invariant line. 
\item A complete control of the multiplicity of the type if the invariant line is unique.  
\item The infinitesimal character is  $(s+1,s,s-1,u)$, for some $s,u\in\mathbb C$. 
\end{itemize} 

\smallskip 
 We now explain how to use these bullets to compute $\Theta(\rho(\chi))$. 
By the remark following Theorem \ref{T:main_2},   $L(\pi(\chi),1)$ is a subquotient of $\Theta(\chi)$. Assume that $\chi$ is not quadratic.
 Since $L(\pi(\chi),1)$ is isomorphic to a degenerate principal series by Proposition \ref{P:degenerate},  
we have a complete control of its $K$-types: the multiplicity is equal to the dimension of $\SO_3$-invariants. Hence $L(\pi(\chi),1)$ accounts for all types 
in $\Theta(\rho(\chi))$ with one-dimensional space of $\SO_3$-invariants. Thus, in order to finish, it suffices to show that every irreducible representation $\pi$ of $G$ 
at  the infinitesimal character of $\Theta(\chi)$ either has a type with zero-dimensional $\SO_3$-invariants, so $\pi$ cannot be a subquotient of  $\Theta(\rho(\chi))$, or has a type 
with one-dimensional $\SO_3$-invariants, but that type has already been occupied by $L(\pi(\chi),1)$.  


\vskip 10pt 
We need the following general result, where $G$ can be a split group corresponding to a root system $\Phi$. 
 Let $P_0=M_0 A_0 N_0$ be a minimal parabolic. Fix a character $\delta$ of $M_0$, and a character  $\lambda$ of $\mathrm{Lie}(A)$.  
 Let $\Ind_{P_0}^G(\delta, \lambda)$ be the corresponding principal series. For any admissible representation $\pi$ of $G$, let 
 $\chi_{\pi}$ denote the formal sum of  its $K$-types.

\begin{prop}\label{P:types} Let $\Re(\lambda)$ be the real part of $\lambda$.
Irreducible subquotients of $\Ind_{P_0}^G(\delta, \Re (\lambda))$ can be partitioned in classes $S_{\pi}$ parametrized  by 
the irreducible subquotients $\pi$ of $\Ind_{P_0}^G(\delta, \lambda)$ such that 
\[ 
\chi_{\pi} = \sum_{\pi' \in S_{\pi}} \chi_{\pi'}. 
\] 

\end{prop} 
\begin{proof} 
Let $\Im(\lambda)$ be the imaginary part of $\lambda$. Assume that $\Im(\lambda)$ is in the closure of the positive chamber corresponding to $P_0$. 
Thus $\Im(\lambda)$ defines a standard parabolic $P=MAN$ such that the roots of $M$ are precisely those perpendicular to $\Im(\lambda)$. Another 
way to define these roots are the roots $\alpha$ such that $(\lambda, \alpha^{\vee}) \in \mathbb R$.   Then, by Lemma \ref{L:vogan}, 
irreducible subquotients  of $\Ind_{P_0}^G(\delta, \lambda)$ are induced from irreducible subquotients of $\Ind_{P_0}^P(\delta, \lambda)$.  
On the other hand, observe that $\Ind_{P_0}^P(\delta, \Re(\lambda))$ is a (positive) character twist of $\Ind_{P_0}^P(\delta, \lambda)$. Thus we have a 
natural correspondence  between irreducible subquotients of these two $P$-modules, preserving types. If $\pi$ is induced from $\sigma$, a subquotient of 
$\Ind_{P_0}^P(\delta, \lambda)$, then we define $S_{\pi}$ to be the (multi) set of irreducible subquotients of $\Ind_P^G(\sigma')$ where $\sigma'$ is the
irreducible subquotient of $\Ind_{P_0}^P(\delta, \Re(\lambda))$ corresponding to $\sigma$.

\end{proof} 

If $\lambda =(s+1,s,s-1,u)$ with $s$ and $u$ imaginary, then $\Re(\lambda)=(1,0,-1,0)$ so the types of representations for these $\lambda$ are 
controlled by the types at a single infinitesimal character $(1,0,-1,0)$.  This will be used in the proof of the main result: 

\smallskip 
\begin{thm} Assume $E=\mathbb R^3$ and $J=M_3(\mathbb R)$, so $T^{\circ}_{E,K} \cong (\mathbb R^{\times})^3/\Delta(\mathbb R^{\times})$. 
 Let $\chi$ be a unitary character of $T^{\circ}_{E,K}$. 
\begin{enumerate} 
\item If $\chi^2\neq 1$ then $\Theta(\rho(\chi)) \cong L(\pi(\chi),1)$. 
\item If $\chi^2=1$ but $\chi\neq 1$ then $\Theta(\rho(\chi)^{\pm}) \cong L(\pi(\chi)^{\pm},1)$.  
\item $\Theta(\rho(1))\cong L(\pi(1),1)$ and $\Theta(\epsilon)\cong \Sigma$, irreducible with the minimal $K$-type $V(1,1,1,1)$.
\end{enumerate} 

\end{thm} 
\begin{proof} We start with (3). Using the {\tt atlas} we can list all irreducible representations with the infinitesimal character 
$(1,1,0,0)$ and trivial central character, and check that every irreducible representation of infinitesimal character  $(1,1,0,0)$ has a type $V(a,b,c,d)$ such that $abcd=0$. 
These types appear in $\Theta(\rho(1))$ and $\Theta(\epsilon)$ with multiplicity 0 or 1, as given by Corollary \ref{C:m_two}. Moreover, these multiplicities coincide with the 
multiplicities in $L(\pi(1),1)$ and $\Sigma$, respectively.  Hence $\Theta(\rho(1))\cong L(\pi(1),1)$ and $\Theta(\epsilon)\cong \Sigma$, since there is no room for anything else. 
(2) is a similar explicit check. For (1) we know that $\Theta(\rho(\chi)) \supseteq L(\pi(\chi),1)$ and that $L(\pi(\chi),1)$ accounts for all types 
$V(a,b,c,d)$ such that $abcd=0$. But, by Proposition \ref{P:types}, any irreducible representation with the same 
infinitesimal character has a type $V(a,b,c,d)$ such that $abcd=0$ (since it is true for $\Re(\lambda)=(1,1,0,0))$, hence $\Theta(\rho(\chi)) = L(\pi(\chi),1)$.   
\end{proof} 

\subsection{Non-split case}

Assume that $E=\mathbb R\times \mathbb C$, 
 In this case $T^{\circ}_{E,K}\cong  (\mathbb R^{\times} \times \mathbb C^{\times} )/\Delta(\mathbb R^{\times})$ and a character 
$\chi$ is a pair $(\chi_{\mathbb R},\chi_{\mathbb C})$ of characters of $\mathbb R^{\times}$  and $\mathbb C^{\times}$ 
such that $\chi_{\mathbb R} \cdot \chi_{\mathbb C} =1$ on $\Delta(\mathbb R^{\times})$. Let
$\nu_1\in \mathbb C$ be the differential of $\chi_{\mathbb R}$. The differential of $\chi_{\mathbb C}$ is a pair of complex numbers 
$(\nu_2,\nu_3)$  such that $\nu_2-\nu_3=m\in \mathbb Z$. Then $(\nu_1,\nu_2,\nu_3)$, $\nu_1 + \nu_2 + \nu_3=0$, 
is the infinitesimal character of $\chi$ and the lift $\Theta(\chi)$ has the infinitesimal character 
\[
 \beta + \frac{\nu_1}{2} \alpha_1 + \frac{\nu_2}{2} \alpha_2 + \frac{\nu_3}{2} \alpha_3. 
\]
 The character $\chi$ is unitary if and only if  $\nu_1$ is purely imaginary,

\smallskip

The group $G=G_E$ is quasi-split. 
Fix a maximal compact subgroup $K$ of $G$. It is isomorphic  $(\Sp(2) \times \SU_2 )/\mu_2$. Fix such an isomorphism. Then
any irreducible representation of $K$ is isomorphic to  $V_{(x,y)} \otimes V_z$ where $x+y+z$ is even. 
The center  $Z_K$ of $K$ is $\mu_2^2/\mu_2$. 
Let $K_J \supseteq K$ be a maximal compact subgroup of $H_J$, and fix an isomorphism 
 $K_J\cong \Sp(4)/\mu_2 \times \mathbb Z/2\mathbb Z$.  Up to $\Sp(4)$-conjugation, the embedding of $K$ into $K_J$ is given by a sequence 
 \[ 
 \Sp(2) \times \SU_2 \subseteq  \Sp(2) \times (\Sp(1) \times \Sp(1) )\subseteq \Sp(4) 
 \] 
 where $\SU_2 \cong \Sp(1)$ is embedded diagonally into $\Sp(1) \times \Sp(1)$. 
 An easy consequence of Lemma \ref{L:branching_1} is that $K$ fixes a line $\mathfrak a$ in $\mathfrak p_J$. Let $A=\exp(\mathfrak a)$. 
 The centralizer of $K$ in $K_J$ is (isomorphic to) $\O_2 \times \mathbb Z/2\mathbb Z$. Let 
 $K'\cong \O_2$ be a subgroup consisting of pairs $(g, \epsilon)$ such that $\det(g)=\epsilon$. Then 
 $T_{E,K}=K' A$, while $T^{\circ}_{E,K}=\SO_2 A$. 
 
 \smallskip 
 Representations of $G$ appearing in the correspondence are 
constituents of the degenerate principal series, obtained by inducing one-dimensional representations of a parabolic group 
$P=MA_PN$ of $G$ such that the derived group of $M$ is a simply connected group corresponding to the $A_2$ root subsystem. 
In particular, we can identify this group with $\SL_3(\mathbb R)$. It is not difficult to check that $M\cong \SL_3(\mathbb R) \cdot \SO_2$. 
There exists an isomorphism $i: \SO_2 A_P \rightarrow \SO_2A$ such that the following holds, as in Proposition \ref{P:degenerate}:

\begin{prop} \label{P:degenerate2} 
Let $\chi$ be a unitary character of $\SO_2A$ and let $\chi_P$ be the character of $\SO_2 A_P$ obtained by pulling back $\chi$ by $i$.  Extend $\chi_P$ to a character of 
$MA_P$ trivial on $\SL_3(\mathbb R)$. 
Let $\Ind_P^G(\chi_P)$ be the unitarizable degenerate principal series representation obtained by inducing $\chi_P$. 
\begin{itemize} 
\item If $\chi\neq 1$,  then $\Ind_P^G(\chi_P) \cong L(\pi(\chi), 1)$. 
\item If $\chi=1$, then $\Ind_P^G(1)\cong L(\rho(1),1) \oplus \Sigma$ where $\Sigma$ is irreducible with the minimal $K$-type $V_{(2,0)} \otimes V_0$. 
\end{itemize} 
\end{prop}

Let $K_M\cong \SO_3 \cdot \SO_2$ be a maximal compact subgroup of $M$. Pick $K\supset K_M$ and fix an isomorphism $K\cong (\Sp(2) \times \SU_2) /\mu_2$. 
Now the embedding of $K_M$ into $K$ is described as follows. The factor  $\SO_3\cong \SU_2/\mu_2$  maps into $(\Sp(2) \times \SU_2 )/\mu_2$ by a composite of embeddings 
   \[ 
  \SU_2 \subseteq  \SU_2 \times \SU_2 \times \SU_2\subseteq \Sp(2) \times \SU_2,
 \] 
 where the first embedding is diagonal. The group $\SO_2$ is contained in the centralizer $(= \O_2)$ of $\Sp(1)$ embedded diagonally into $\Sp(2)$.
  Consider the degenerate principal series $\Ind_P^G(\chi_P)$. 
 Let $\chi_m$ be the restriction of $\chi_P$ to character to $\SO_2$, where $m\in \mathbb Z$ and we have identified the dual of $\SO_2$ with 
 $\mathbb Z$. By the Frobenius reciprocity, the multiplicity of the $K$-type $V_{(x,y)} \otimes V_z$ in the degenerate principal series is the same as the multiplicity of 
  the $\Sp(1) \times \SO_2$-module  $V_z \otimes \chi_m$ in the $\Sp(2)$-module $V_{(x,y)}$.  If we ignore the factor $\SO_2$, a type $V_{(x,y)}\otimes V_z$ 
  appears only if it has non-zero $\SO_3$-fixed vectors, and this may happen only if 
  \[ 
  x+ y \geq z. 
  \] 

  \smallskip 
 Let $\chi$ be a character of $T_{E,K}^{\circ}=\SO_2A$ and let $\chi_m$ be the restriction of $\chi$ to $\SO_2$. Let $V[m]$ be the summand of the minimal 
 representation on which this $\SO_2$ acts by $\chi_m$.  Let $\Theta(\chi)$ be the full lift of $\chi$ to $G$. 
 Then $\Theta(\chi)$ is a quotient of $V[m]$. We shall prove that the $K$-types and their multiplicities of 
$\Theta(\chi)$ are less than or equal to those of the degenerate principal series, as just discussed. The key is the following two lemmas.

\begin{lemma} \label{L:bounded_1}  Let $d$ be the multiplicity of the $\Sp(1) \times \SO_2$-module $V_z \boxtimes \chi_m$ in the 
$\Sp(2)$-module $V_{(x,y)}$.  Then
\[ 
 \Hom_K(V_{(x,y)} \otimes V_z, V_{i\omega_4} [m]) 
 \]
 is trivial if $i<x$ and equal to $d$ if $i\geq x$. 
\end{lemma} 
\begin{proof} This follows from Lemma \ref{L:branching_1}. 
\end{proof}

\begin{lemma} \label{L:linear} 
 Let $X$ be a torsion free $\mathbb C[x]$-module. Assume that $X$ is a union finite dimensional spaces $X_1 \subseteq X_2 \subseteq \ldots$ such that 
\begin{itemize} 
\item $x \cdot X_n \subset X_{n+1}$ for all $n$. 
\item $\dim(X_{n+1}/X_n) = d$ for all sufficiently large $n$. 
\end{itemize} 
Then for every $\nu\in \mathbb C$, $\dim(X/(x-\nu)X) \leq d$. 
\end{lemma} 
\begin{proof} Assume otherwise, and pick $d+1$ linearly independent functionals on $X$ vanishing on $(x-\nu)X$.  Pick a finitely generated submodule $F$ 
such that the restrictions to $F$ of these linear functionals are linearly independent. Hence $\dim(F/(x-\nu)F) > d$. On the other hand, since $F$ is  
finitely generated, it is free, hence $F\cong \mathbb C[x]^r$ for some $r$, and $r>d$.  Let $k$ be such that $X_k$ contains all generators of  $F$.  
Let $F_n = F \cap X_{n+k}$. Then $\dim F_n \geq r n$. On the other hand, $\dim X_{n+k} =dn +c$ for some constant $c$, independent of $n$.  This contradicts 
that $F_n\subseteq X_{n+k}$ for all $n$ 
\end{proof}

   \smallskip 

   \begin{cor}  \label{C:m_one_2} Let $\chi$ be a character $T_{E,K}^{\circ}$. 
   The multiplicities of $K$-types in $\Theta(\chi)$ are less than or equal to those of the degenerate principal series  $\Ind_P^G(\chi_P)$. The equality holds for types of multiplicity one. 
    \end{cor} 
   \begin{proof} Assume that the restriction of $\chi$ to the maximal compact subgroup is $\chi_m$. 
    Recall that the multiplicity $d$ of $V_{(x,y)} \otimes V_z,$ in $\Ind_P^G(\chi_P)$  is equal to the 
   multiplicity $d$ of the $\Sp(1) \times \SO_2$-module $V_z \boxtimes \chi_m$ in the $\Sp(2)$-module $V_{(x,y)}$. On the other hand, 
   the multiplicity of $V_{(x,y)} \otimes V_z,$ in $\Theta(\chi)$ is equal to $X/(x-\nu)X$ where 
\[ 
X=\oplus_{i=0}^{\infty}  \Hom_K(V_{(x,y)} \otimes V_z, V_{i\omega_4} [m]) 
\] 
and $\nu=d\chi(x)$, $U(\mathfrak a ) \cong \mathbb C[x]$.  It follows from Lemma   \ref{L:linear} (its conditions are satisfied by Lemma \ref{L:bounded_1}) 
that the multiplicity is bounded by $d$.  If $d=1$, then $X$ is clearly a free 
$U(\mathfrak a)$-module of rank one, so the type appears in $\Theta(\chi)$ with multiplicity one.  
\end{proof} 
 
 \smallskip 
\begin{thm}  Assume $E=\mathbb R\times\mathbb C$ and $J=M_3(\mathbb R)$, so $T^{\circ}_{E,K}\cong \mathbb C^{\times}$. 
Let $\chi$ be a unitary character of $T^{\circ}_{E,K}$. 
\begin{enumerate} 
\item If $\chi\neq 1$ then $\Theta(\rho(\chi)) \cong L(\pi(\chi),1)$. 
\item $\Theta(\rho(1))\cong L(\pi(1),1)$ and $\Theta(\epsilon)\cong \Sigma$,  irreducible with the minimal $K$-type $V_{(2,0)}\otimes V_0$. 
\end{enumerate} 

\end{thm} 
\begin{proof} The torus $T_{E,K}$ has only two one-dimensional representations, the trivial $\rho(1)$ and the sign $\epsilon$, and the family of 2-dimensional representations 
   $\rho(\chi)\cong\rho(\chi^{-1})$ parameterized by non-trivial characters $\chi$. Hence $\Theta(\rho(\chi))\cong \Theta(\chi)$ if $\chi\neq 1$ and 
   $\Theta(1)\cong \Theta(\rho(1))\oplus \Theta(\epsilon)$ otherwise. By Corollary \ref{C:m_one_2}, 
   if $\chi\neq 1$, the types of $\Theta(\rho(\chi))$ are 
   bounded by the types of $\Ind_P^G(\chi_P)\cong L(\pi(\chi),1)$. Since, 
   by the remark following Theorem \ref{T:main_2},   $L(\pi(\chi),1)$ is a subquotient of $\Theta(\chi)$ it follows that
   $\Theta(\rho(\chi))\cong L(\pi(\chi),1)$ for $\chi\neq 1$.  This takes care of (1), so we look at (2). We need the following 
   
  \begin{lemma} \label{L:m_two} The type $V_{(2n,0)}\otimes V_0$ appears  in $\Theta(\rho(1))$ if $n$ is even and 
in $\Theta(\epsilon)$ if $n$ is odd. The multiplicity is always 1. 
\end{lemma} 
\begin{proof} This type appears in  $V_{i\omega_4}[0]$ for $i=2n, 2n+1, \ldots $, with multiplicity one. 
Thus, as in the proof of Corollary \ref{C:m_two},  we need to figure out the action of the factor $\mathbb Z/2\mathbb Z$ on the type, for $i=2n$, to find out if it 
belongs to $\Theta(\rho(1))$ or $\Theta(\epsilon)$. This reduces to finding the sign of the action of $\O_2/\SO_2$ on 
$V_{(2n,0)}^{\Sp(1)\times \SO_2}$. But this can be easily computed by observing that $V_{(2n,0)}$ is $2n$-th symmetric power of $V_{(1,0)}$. 
The answer is $(-1)^n$. 
\end{proof} 

The lemma implies that $\Theta(\rho(1))$ and $\Theta(\epsilon)$  are non-zero, and the former is spherical. 
Here we have 11 representations with this infinitesimal character, see Section \ref{SS:quasi_split}. Only two, $L(\pi(1),1)$ and $\Sigma$ have the $K$-types 
in the cone $x+y \geq z$. Hence $\Theta(\rho(1)) \cong L(\pi(1),1)$ and  $\Theta(\epsilon)\cong \Sigma$, since there is no 
room for anything else. 
\end{proof}

\section{The correspondences in the case $J=H_3(\mathbb C)$} 

There are two cases here. If $E=\mathbb R^3$, then $T_{E,K}$ is compact. This case is treated in \cite{Lokethesis} and we shall include here that 
result in a suitable form. 

\subsection{Minimal representation} 
Let $K_J$ be a maximal compact subgroup of $H_J$. 
The  group $K_J$ is isomorphic to $\SU_2\times_{\mu_2}\SU_6 \rtimes \mathbb Z/2\mathbb Z$. More precisely, since we work with adjoint groups, the quotient 
$\SU_6/\mu_3$ appears instead of $\SU_6$, however for ease of notation, we shall often write $\SU_6$. 
The generator $\delta$ of $\mathbb Z/2\mathbb Z$ acts on $\SU_6$ as an outer 
automorphism.  Let $\omega_3$ be the third fundamental weight for $\SU_6$. Let $V_n$, as previously, denote the irreducible representation of $\SU_2$ of highest weight $n$. 
 Let $V$ be the $(\mathfrak h_J, K_J)$-module corresponding to the minimal representation of $G_J$. We have an isomorphism of $\SU_2\times \SU_6$-modules \cite{GrW},
\[ 
V= \oplus _{n=0}^{\infty} V_{n+2}\otimes V_{n\omega_3} . 
\] 
The element $\delta$ acts on each $V_{n\omega_3}$, and we now describe its action. 
The centralizer of $\delta$ in $\SU_6$ is $\Sp(3)$ so its eigenspaces in $V_{n\omega_3}$ are naturally $\Sp(3)$-modules. 
In fact, $\SU_2\times_{\mu_2} \Sp(3)$ is a maximal compact subgroup 
of the split Lie group of the type $F_4$, and understanding the eigenspaces  of $\delta$ in $V$ amounts restricting $V$ to $F_4$. This was done in 
\cite{LokeJFA} and, as a consequence, we have the following. 
The restriction of  $V_{n\omega_3}$ to $\Sp(3)$ is the multiplicity free sum of all irreducible representations $V_{(n,m,m)}$ with the highest 
weight $(n,m,m)$, where $n\geq m \geq 0$, and $\delta$ acts on $V_{(n,m,m)}$ by $(-1)^{n-m}$. 

\subsection{Split case} 
Assume that $E=\mathbb R^3$. 
 In this case $T^{\circ}_{E,K}\cong \mathbb T^3/\Delta\mathbb T$, and a character $\chi$ is a given triple of integers $(\nu_1, \nu_2, \nu_3)$ 
 such that $\nu_1 + \nu_2 + \nu_3=0$. This triple is the infinitesimal character of $\chi$.  
 Thus the lift $\Theta(\chi)$ has the infinitesimal character 
\[
 \beta + \frac{\nu_1}{2} \alpha_1 + \frac{\nu_2}{2} \alpha_2 + \frac{\nu_3}{2} \alpha_3. 
\]
 
\smallskip 

The group $G=G_E$ is split. 
Let $K$ be the maximal compact subgroup of $G$. Recall that $K\cong \SU_2^4/\mu_2$.  
We have 4 embeddings $\mathbb R^3 \rightarrow H_3(\mathbb C)$, which give rise to 
4 non-conjugated embeddings of $G$ into $H _J$. On the level of maximal compact groups the 4 different embeddings are easily distinguished. We have 
4 embeddings $K\rightarrow K_J$ each obtained by picking one $\SU_2$-factor of $K$ and  mapping it to the unique $\SU_2$ factor of $K_J$. The other three $\SU_2$ are 
mapped into $\SU_6$, uniquely up to conjugation in $\SU_6$. So we fix one embedding of $K$ into $K_J$ and assume, without loss of generality, that the first $\SU_2$ factor 
of $K$ maps into the $\SU_2$-factor of $K_J$. We note that $V$ is $\SU_2$-admissible, for this $\SU_2$. Hence the restriction of $V$ to $G$ is also admissible, and this 
dual pair correspondence was determined in \cite{Lokethesis}. 
 Each $\Theta(\rho)$ is an irreducible quaternionic representation, which can be described as a derived functor module. 
 Here we give a reinterpretation of that result in terms of Langlands parameters. 
 
 \begin{thm}  Assume $E=\mathbb R^3$ and $J=H_3(\mathbb C)$, so $T^{\circ}_{E,K} \cong \mathbb T^3/\Delta\mathbb T$. 
 Let $\chi$ be a character of $T^{\circ}_{E,K}$, and 
 $P(E,K,\chi)$ be the packet of 4 representations of the Levi $M_E$ as in Section \ref{packets}. 
\begin{enumerate} 
\item For every $\chi$, $\Theta(\rho(\chi)) \cong L(\pi,1)$, for $\pi \in P(E,K,\chi)$. 
\item $\Theta(\rho(1))$ and $\Theta(\epsilon)$ is irreducible representations of $G_E$ with the infinitesimal character $(1,1,0,0)$ and minimal $K$-types $V(2,0,0,0)$ and 
$V(4,0,0,0)$. 
\end{enumerate} 

\end{thm} 

Of course, if we run through all 4 embeddings of $\mathbb R^3$ into $H_3(\mathbb C)$ then $\pi$ runs through all of $P(E,K,\chi)$. The minimal types of 
$\Theta(\epsilon)$ are $V(4,0,0,0)$, $V(0,4,0,0)$ $V(0,0,4,0)$ and $V(0,0,0, 4)$.

\subsection{Quasi-split case} 
Assume that $E=\mathbb R\times \mathbb C$, 
 In this case $T^{\circ}_{E,K}\cong  (\mathbb T \times \mathbb C^{\times} )/\Delta(\mathbb T)$ and a character 
$\chi$ is a pair $(m,\chi_{\mathbb C})$, where $m\in \mathbb Z$, such that the restriction of $\chi_{\mathbb C} =1$ on $\mathbb T$ is 
given by $z\mapsto z^{-m}$. In this case the differential of $\chi_{\mathbb C}$ is a pair of complex numbers 
$(\nu_2,\nu_3)$  such that $\nu_2+\nu_3=-m$. Then $(m,\nu_2,\nu_3)$, 
is the infinitesimal character of $\chi$ and the lift $\Theta(\chi)$ has the infinitesimal character 
\[
 \beta + \frac{m}{2} \alpha_1 + \frac{\nu_2}{2} \alpha_2 + \frac{\nu_3}{2} \alpha_3. 
\]
 The character $\chi$ is unitary if and only if $\nu_2-\nu_3$ is purely imaginary. 
 
 \smallskip 

In this case $G=G_E$ is quasi-split. 
Let $K$ be the maximal compact subgroup of $G$. Recall that $K\cong \SU_2\times_{\mu_2}\Sp(2)$.   The embedding of $K$ into $K_J$ is given by a sequence 
\[ 
\SU_2\times \Sp(2) \subseteq \SU_2 \times \Sp(2)\times \SU_2 \subseteq \SU_2 \times \SU(6) 
\] 
where the first inclusion is identity on $\Sp(2)$ and the diagonal embedding of $\SU_2$ into two $\SU_2$. The second embedding is identity on the first $\SU_2$ 
 and the unique embedding of $\Sp(2) \times \SU_2 \subseteq \SU_6$. The maximal compact subgroup of $T^{\circ}_{E,K}$ commutes with $K$, so its Lie algebra is 
  spanned by 
 \[ 
 h = \frac{1}{3}(1,1,1,1,-2,-2), 
 \] 
 thought of as a diagonal matrix of $\mathfrak{sl}_6(\mathbb C)$, the Lie algebra of $\SU_6/\mu_3$. This choice of the generator of the Lie algebra fixes an isomorphism 
 of the maximal compact subgroup of $T^{\circ}_{E,K}$  with $\U_1 \subseteq \SU_6/\mu_3$. The maximal compact subgroup of $T_{E,K}$ is generated by 
 $U(1)$ and $\delta$. 
 
 \smallskip 
We now describe the representations of $G$ appearing in the correspondence. To that end, 
let $P=MAN$ be the maximal parabolic subgroup of $G$ such that the connected component of $M^{\circ}$ of $M$ is  
isomorphic to  $\SU_{2,2}(\mathbb R)$.  The maximal compact subgroup of $\SU_{2,2}(\mathbb R)$ is isomorphic to 
$(\SU_2 \times \SU_2) \times_{\mu_2} \U_1$, and we fix such an isomorphism. 
Assume $m\geq 0$.  Let $\sigma_m$ be the lowest weight  $\mathfrak{su}_{2,2}$-module with the minimal type  
$V_0\otimes V_{0+m}[m+2]$, and the lowest $\U_1$-weight $m+2$ as the notation indicates. 
We have the following:

\begin{prop} \label{P:degenerate3} 
Let $\chi$ be a unitary character of $T^{\circ}_{E,K}$ such that the restriction to $\U_1$ is $\chi_m$. 
Let $\Ind_{P^{\circ}}^G(\sigma_m \otimes e^{\lambda})$ be the unitarizable principal series representation obtained by inducing $\sigma_m \otimes e^{\lambda}$ where 
connected component $P^{\circ}$ of $P$, and $\lambda \in  \mathfrak a^{\ast}_{\mathbb C}$.  
\begin{itemize} 
\item If $\chi\neq 1$,  then $\Ind_P^G(\sigma_m \otimes e^{\lambda}) \cong L(\pi(\chi), 1)$, for $\lambda$ depending on $\chi$. 
\item If $\chi=1$, then $\Ind_P^G(\sigma_0 )\cong L(\rho(1),1) \oplus \Sigma$ where $\Sigma$ is irreducible with the minimal $K$-type $V_{(0,0)} \otimes V_4$. 
\end{itemize} 
\end{prop}

 Proposition \ref{P:degenerate3} gives us a control of the $K$-types of $L(\pi(\chi), 1)$, since we can compute the $K$-types of principal series representations: 
  
  \begin{prop}  \label{P:multiplicities_1} 
   If $V_{(x,y)} \otimes V_z$ is a $K$ type in $\Ind_{P^{\circ}}^G( \sigma_m \otimes e^{\lambda})$ then $z\equiv x-y \equiv m \pmod 2$ and 
  $z> x-y \geq m$. If that is the case,  the multiplicity of the type is equal to the number of integers $t$ such that 
  \[ 
  x+y -m \geq 2t \geq x-y-m \text { and } z \geq 2t+2+m.   
  \] 
  
  \end{prop}  
 \begin{proof} 
 The $\SU_2 \times \SU_2 \times \U_1$ types of $\sigma_m$ are 
 \[ 
  V_0\otimes V_{0+m}[m+2] + V_1\otimes V_{1+m}[m+4] + \ldots 
 \]  
 Counting multiplicities of these types in $V_{(x,y)} \otimes V_z$ gives the multiplicity of 
 $V_{(x,y)} \otimes V_z$ in $\Ind_{P^{\circ}}^G( \sigma_m \otimes e^{\lambda})$. 
 
 We now count $t$ such that $V_t \otimes V_{t+m} [2t+2+m]$ is contained in 
 $V_{(x,y)} \otimes V_z$. Clearly $2t+2+m\leq z$ is a necessary condition, along with $z\equiv m \pmod 2$.  
  Next, restrict $V_{(x,y)}$ to $\SU_2 \times \SU_2$.  Then $V_t\otimes V_{t+m}$ is contained in $V_{(x,y)}$ if and only if the inequalities of Lemma \ref{L:branching_2} 
  are satisfied. But these are exactly the same as claimed. 
\end{proof}

 \smallskip

 We now need the branching rule for the restriction of 
 $V_{n\omega_3}$ from $\SU_6/\mu_3$ to $(\Sp(2) \times \SU_2) \times \U_1$ where, recall, $\U_1$ is the maximal compact subgroup of  $T^{\circ}_{E,K}$. 
 Let $V_{n\omega_3}[m]$  be the subspace such that $\U_1$ acts as the character $\chi_m$. 
 The following is a consequence of the Gelfand-Zetlin branching,  applied twice to obtain a branching from $\SU_6$ to $\SU_4 \times \SU_2$, 
 combined with the restriction from $\SU_4$ to $\Sp(2)$, see \cite{HTW}. 
 
 \begin{lemma}  Assume $m\geq 0$. 
  The restriction of $V_{n\omega_3}[m]$ to $\Sp(2) \times \SU_2$   is  
 \[ 
V_{n\omega_3}[m]=  \sum_{t\geq 0}( \sum V_{(x,y)}) \otimes V_{n-m-2t}  
 \] 
 where the inner sum is over all $ V_{(x,y)}$ such that $x+y \equiv m\pmod{2}$ and 
 \[ 
 2n-2t -2m\geq x+y -m\geq 2t \geq x-y -m \geq 0. 
 \] 
\end{lemma}

Let $\chi$ be a character of $T_{E,J}^{\circ}\cong \mathbb C^{\times}$ and let $\chi_m$ be the restriction of $\chi$ to $\U_1$. Let $V[m]$ be the summand of the minimal 
 representation on which $\U_1$ acts by $\chi_m$.  Let $\Theta(\chi)$ be the full lift of $\chi$ to $G$. 
 Note that $\Theta(\chi)$ is a quotient of $V[m]$. 
 
  \begin{prop} \label{P:multiplicities_2} 
  Assume $m\geq 0$. The multiplicity of $V_{(x,y)} \otimes V_z$  in $V_{n\omega_3}[m] \otimes V_{n+2}$ is 0 unless  $z> x-y\geq m$ and $z\equiv x-y\equiv m \pmod{2}$. In that case 
 the multiplicity is increasing as $n$ increases and stabilizes at the number of  $t\in \mathbb Z$ such that 
 \[ 
 x+y -m\geq 2t \geq x-y -m\geq 0 \text{ and }  z \geq m+ 2t +2. 
 \] 
 \end{prop} 
 \begin{proof}   By the previous lemma the multiplicity is equal to the number of $t$ such that 
  \[ 
 2n-2t-2m \geq x+y -m  \geq 2t \geq x-y-m  \geq 0
 \] 
 and $V_z$ is a summand in the tensor product 
\[ 
 V_{n+2} \otimes V_{n-m-2t} = V_{2n-m-2t+2} \oplus \ldots \oplus V_{m+2t +2}. 
 \] 
 Thus $t$ has to satisfy the additional inequalities 
 \[
 2n-m-2t + 2\geq z \geq m+ 2t +2.  
 \] 
 The proposition follows by combining the two sets of inequalities and taking $n$ large enough. 
  \end{proof}

  In view of Proposition \ref{P:multiplicities_2}, we can apply  Lemma \ref{L:linear} to the $U(\mathfrak a)$-module 
  \[ 
 X= \oplus_{n=0}^{\infty}\Hom_K ( V_{(x,y)} \otimes V_z, V_{n\omega_3}[m] \otimes V_{n+2}) 
  \] 
  giving an upper bound on the types of $\Theta(\chi)$. That bound is the same as 
   the multiplicities of types in $V_{(x,y)} \otimes V_z$ in $\Ind_{P^{\circ}}^G(\sigma_m \otimes e^{\lambda})$, 
   given by Proposition \ref{P:multiplicities_1}, thus we have the following: 
  
    \begin{cor}  \label{C:m_one_3} Let $\chi$ be a character $T_{E,K}^{\circ}$ whose restriction to the maximal compact subgroup is $\chi_m$. Then the $K$-types and 
   their multiplicities in $\Theta(\chi)$ are less than or equal to those of the degenerate principal series  $\Ind_{P^{\circ}}^G(\sigma_m \otimes e^{\lambda})$. The equality holds for types of multiplicity one. 
    \end{cor}

 \smallskip 
\begin{thm}  Assume $E=\mathbb R\times\mathbb C$ and $J=H_3(\mathbb C)$, so $T^{\circ}_{E,K}\cong \mathbb C^{\times}$. 
Let $\chi$ be a unitary character of $T^{\circ}_{E,K}$. 
\begin{enumerate} 
\item If $\chi\neq 1$ then $\Theta(\rho(\chi)) \cong L(\pi(\chi),1)$. 
\item $\Theta(\rho(1))\cong L(\pi(1),1)$ and $\Theta(\epsilon)\cong \Sigma$ the unique irreducible with the minimal $K$-type 
 $V_{(0,0)}\otimes V_4$ and the infinitesimal character $(1,1,0,0)$. 
\end{enumerate} 
\end{thm}
\begin{proof}  If $\chi\neq 1$ then $\Theta(\chi) \cong \Theta(\chi^{-1}) \cong \Theta(\rho(\chi))$. Since $\Theta(\rho(\chi))$ contains $L(\pi(\chi),1)$, 
by the remark following Theorem \ref{T:main_2}, 
(1) follows at once from Corollary \ref{C:m_one_3}. For (2) we need the following lemma. 

\begin{lemma}  \label{L:m_three}
The type $V_{(0,0)}\otimes V_{2n}$, $n>0$, appears in $\Theta(\rho(1))$ for $n$ odd, and 
in $\Theta(\epsilon)$ for $n$ even, with multiplicity one. 
  \end{lemma} 
  \begin{proof} 
  The $K$ types trivial on $\Sp(2)$, i.e. $x=y=0$, appear in $V[m]$ only for $m=0$. Hence, we can restrict $V$, not just $V[0]$, to compute the types of 
  $\Theta(\rho(1))$ and $\Theta(\epsilon)$.  So we take the $n$-type $V_{n+2} \otimes V_{n\omega_3}$ and restrict it to $K=\Sp(2)\times \SU_2$. 
  We shall use that the restriction of $V_{n\omega_3}$ to $\Sp(3)$ is the sum of $V_{(n,m,m)}$ and that 
  $\delta\in \mathbb Z/2\mathbb Z$ acts on $V_{(n,m,m)}$ by $(-1)^{n-m}$. By the branching from $\Sp(3)$ to $\Sp(2)\times \Sp(1)$, the type trivial 
  on $\Sp(2)$ appears only for $m=0$ and with $V_n$ on the $\Sp(1)$-factor.   
  Note that $\delta$ acts on it as $(-1)^n$. Using the Clebsch-Gordan formula, we see that the $K$-type $V_{(0,0)}\otimes V_{2n+2}$ appears the first time in the 
  $n$-th $K_J$-type of the minimal representation, and $\delta$ acts on it by $(-1)^n$. Whether the type appears in $\Theta(\rho(1))$ or 
  $\Theta(\epsilon)$  depends precisely on this sign, see the proof of Corollary \ref{C:m_two}. The lemma is proved. 
  \end{proof} 

The above Lemma implies that $\Theta(\epsilon)\neq 0$ and contains the type $V_{(0,0)}\otimes V_4$. We also know that 
 $\Theta(\rho(1)) \supseteq L(\pi(1),1)$ and we have a bound on all types of the sum 
$\Theta(1)=\Theta(\rho(1))\oplus \Theta(\epsilon)$. Hence the minimal type of  $\Theta(\epsilon)$ is $V_{(0,0)}\otimes V_4$ and therefore contains $\Sigma$. 
 Both lifts are irreducible since there is no room for anything else. 
\end{proof} 

\section{Representations} \label{S:representations} 
We list representations of $\Spin(4,4)$ and $\Spin(5,3)$ with the infinitesimal character $\gamma=(1,1,0,0)$, along with their lowest $K$-types. 
It is a fairly straightforward exercise to do this by hand, using one of many references available, such as \cite{VogUnit}. However, for convenience, we use the {\tt atlas}
software to obtain the needed details. We give some (limited) explanation of the {\tt atlas} notation here and hope most readers will find this sufficient. More details 
can be filled in using the vast amount of information about the software available on the website www.liegroups.org.

Recall that an irreducible admissible representation $\pi$ of a real group $G$ can be given by a Cartan subgroup $H$ and a character $\chi$ of (a cover of) $H$, from which we  can obtain inducing data for a standard module with unique irreducible quotient $\pi$. Correspondingly, in {\tt atlas}, $\pi$ is represented by a ``parameter", 
which is a triple $({\tt x},\lambda,\nu)$. Here ${\tt x}$ is a ``${\tt KGB}$ element", which determines $H$, $\lambda$ is the discrete part of $\chi$ (essentially the Harish-Chandra parameter of a limit of 
discrete series of a Levi subgroup of $G$), and $\nu$ is the continuous part of the character. 

The software uses coordinates different from the standard ones chosen in this paper; in fact, the {\tt atlas} simple roots are given by the columns of the Cartan matrix for Lie type $D_4$. In these
coordinates, the weight $\gamma=(1,1,0,0)$ is given by $(0,1,0,0)$. We give the {\tt atlas} output, then convert the data into our coordinates. This conversion is done using the software as well.

\medskip 

\subsection{Split case} \label{SS:split} 
We limit ourselves to representations of $Spin(4,4)$ with infinitesimal character $\gamma$ and trivial central character. In {\tt atlas} coordinates, their parameters are:

\medskip 

\begin{verbatim}
0 final parameter(x=108,lambda=[1,1,1,1]/1,nu=[0,1,0,0]/1)
1 final parameter(x=108,lambda=[1,2,1,1]/1,nu=[0,1,0,0]/1)
2 final parameter(x=94,lambda=[0,3,0,0]/1,nu=[0,1,0,0]/1)
3 final parameter(x=93,lambda=[0,3,0,0]/1,nu=[0,1,0,0]/1)
4 final parameter(x=92,lambda=[0,3,0,0]/1,nu=[0,1,0,0]/1)
5 final parameter(x=91,lambda=[0,3,0,0]/1,nu=[0,1,0,0]/1)
6 final parameter(x=65,lambda=[-1,4,-1,-1]/1,nu=[-1,3,-1,-1]/2)
7 final parameter(x=52,lambda=[-1,2,1,-1]/1,nu=[0,1,0,-1]/1)
8 final parameter(x=51,lambda=[-1,2,1,-1]/1,nu=[0,1,0,-1]/1)
9 final parameter(x=50,lambda=[-2,3,0,0]/1,nu=[-1,1,0,0]/1)
10 final parameter(x=49,lambda=[-2,3,0,0]/1,nu=[-1,1,0,0]/1)
11 final parameter(x=48,lambda=[0,3,-2,0]/1,nu=[0,1,-1,0]/1)
12 final parameter(x=47,lambda=[0,3,-2,0]/1,nu=[0,1,-1,0]/1)
13 final parameter(x=19,lambda=[0,1,0,0]/1,nu=[-1,2,-1,-1]/2)
14 final parameter(x=18,lambda=[0,1,0,0]/1,nu=[-1,2,-1,-1]/2)
15 final parameter(x=17,lambda=[0,1,0,0]/1,nu=[-1,2,-1,-1]/2)
16 final parameter(x=16,lambda=[0,1,0,0]/1,nu=[-1,2,-1,-1]/2)
17 final parameter(x=9,lambda=[0,1,0,0]/1,nu=[0,0,0,0]/1)
18 final parameter(x=7,lambda=[0,1,0,0]/1,nu=[0,0,0,0]/1)
19 final parameter(x=6,lambda=[0,1,0,0]/1,nu=[0,0,0,0]/1)
20 final parameter(x=5,lambda=[0,1,0,0]/1,nu=[0,0,0,0]/1)
21 final parameter(x=4,lambda=[0,1,0,0]/1,nu=[0,0,0,0]/1)
22 final parameter(x=3,lambda=[0,1,0,0]/1,nu=[0,0,0,0]/1)
23 final parameter(x=1,lambda=[0,1,0,0]/1,nu=[0,0,0,0]/1)
24 final parameter(x=0,lambda=[0,1,0,0]/1,nu=[0,0,0,0]/1)
\end{verbatim}

\medskip 

We translate the parameters into our standard coordinates, and calculate the lowest $K$-types of the corresponding representations. Here $K\cong SU_2^4/\mu_2$, and we give the highest weights with respect to simple roots $2\epsilon_1$, $2\epsilon_2$, $2\epsilon_3$, and $2\epsilon_4$, as well as their dimension. The numbering is the one chosen by {\tt atlas}.

\medskip 

\begin{tabular}{| c |  c | c | c | c | c|}\hline
 \#    & {\tt x} & $\lambda$ & $\nu$ & highest weight & dimension \\ \hline
 0  &  108  &[3,2,1,0] & [1,1,0,0]& [0,0,0,0]&1 \\
 1 & 108 & [4,3,1,0]&[1,1,0,0] &[0,0,0,2 &3 \\
 &&&&[0,0,2,0]&3\\
 &&&&[0,2,0,0]&3\\
 &&&&[2,0,0,0]&3\\
 2 & 94 &[3,3,0,0]& [1,1,0,0]& [2,0,0,0]& 3\\
 3 &  93& [3,3,0,0]& [1,1,0,0]&[0,2,0,0] & 3\\
 4 &92 & [3,3,0,0]&[1,1,0,0] &[0,0,2,0] &3 \\
 5 &  91&[3,3,0,0] &[1,1,0,0] &[0,0,0,2] &3 \\
 6 & 65 &[2,3,-1,0] &[1/2,1,-1/2,0] &[1,1,1,1] & 16\\
 7 & 52 &[1,2,0,-1] & [1/2,1/2,-1/2,-1/2]&[2,2,0,0] & 9\\
 8 & 51 &[1,2,0,-1] &[1/2,1/2,-1/2,-1/2] &[0,0,2,2] & 9\\
 9 &  50& [1,3,0,0]& [0,1,0,0]& [0,2,2,0]& 9\\
 10 & 49 & [1,3,0,0]&[0,1,0,0] & [2,0,0,2]& 9\\
 11& 48 &[2,2,-1,1] & [1/2,1/2,-1/2,1/2]&[2,0,2,0] & 9\\
 12& 47 & [2,2,-1,1]& [1/2,1/2,-1/2,1/2] &[0,2,0,2] &9 \\
 13& 19&  [1,1,0,0]& [0,1/2,-1/2,0]& [4,0,0,0]& 5\\
 14& 18&  [1,1,0,0]&[0,1/2,-1/2,0] & [0,4,0,0]& 5\\
 15& 17&  [1,1,0,0]&[0,1/2,-1/2,0] & [0,0,4,0]& 5\\
 16& 16& [1,1,0,0] &[0,1/2,-1/2,0] &[0,0,0,4] &5 \\
 17&9 & [1,1,0,0] &[0,0,0,0] & [2,2,2,0]& 27\\
 18& 7&  [1,1,0,0]& [0,0,0,0]& [3,1,1,1]& 32\\
 19&6 & [1,1,0,0] & [0,0,0,0]&[1,3,1,1] & 32\\
 20& 5& [1,1,0,0] & [0,0,0,0]& [1,1,3,1]&32 \\
 21&4 & [1,1,0,0] & [0,0,0,0]& [2,2,0,2]& 27\\
 22& 3& [1,1,0,0] & [0,0,0,0]& [2,0,2,2]&27 \\
 23& 1& [1,1,0,0] &[0,0,0,0] &[0,2,2,2] & 27\\
 24&  0 & [1,1,0,0] &[0,0,0,0] & [1,1,1,3]& 32\\ \hline
 
 \end{tabular}

\bigskip
{\tt KGB} elements 0, 1, 3, 4, 5, 6, 7, and 9 correspond to the compact Cartan subgroup. This means that the last eight parameters are those of the limit of discrete
series representations of $Spin(4,4)$ at the given infinitesimal character. {\tt KGB} element 108 belongs to the split CSG, so the first two are
principal series representations, one of them having four lowest $K$-types. Representations 2 through 5 are attached to the Cartan subgroup 
$\mathbb T^2\times\mathbb C^{\times}$, but they may also be realized as the four lowest $K$-type constituents of a third principal series representation. 
The remaining representations are a little more difficult to describe, but we specify
the Cartan subgroups to which they are attached. 
{\tt KGB} elements 16, 17, 18, 19 (as well as 91, 92, 93, and 94) correspond to
the CSG $\mathbb T^2\times\mathbb C^{\times}$. There are three non-conjugate CSGs isomorphic to 
$\mathbb T\times\mathbb C^{\times}\times\mathbb R^{\times}$
with corresponding  {\tt KGB} elements 51 and 52, 49 and 50, and 47 and 48. Finally, {\tt KGB} element 65 corresponnds to the CSG 
$\mathbb C^{\times}\times(\mathbb R^{\times})^2$.

\medskip

\subsection{Quasi-split case} \label{SS:quasi_split}
Here we have 11 representations at infinitesimal character $\gamma$:

\smallskip 

\begin{verbatim}
0 final parameter(x=36,lambda=[1,2,0,0]/1,nu=[0,1,0,0]/1)
1 final parameter(x=36,lambda=[1,3,0,0]/1,nu=[0,1,0,0]/1)
2 final parameter(x=32,lambda=[0,3,0,0]/1,nu=[0,1,0,0]/1)
3 final parameter(x=16,lambda=[0,3,-1,-1]/1,nu=[-1,3,-1,-1]/2)
4 final parameter(x=24,lambda=[-1,1,1,1]/1,nu=[-1,1,0,0]/1)
5 final parameter(x=24,lambda=[-2,2,1,1]/1,nu=[-1,1,0,0]/1)
6 final parameter(x=14,lambda=[-1,2,0,0]/1,nu=[-1,1,0,0]/1)
7 final parameter(x=12,lambda=[-1,2,0,0]/1,nu=[-1,1,0,0]/1)
8 final parameter(x=3,lambda=[0,1,0,0]/1,nu=[-1,2,-1,-1]/2)
9 final parameter(x=1,lambda=[0,1,0,0]/1,nu=[0,0,0,0]/1)
10 final parameter(x=0,lambda=[0,1,0,0]/1,nu=[0,0,0,0]/1)
\end{verbatim}

\medskip

Again, we translate the parameters into our standard coordinates, and calculate the lowest $K$-types of the corresponding representations. Here $K=\Sp(2) \times \SU_2$, with simple roots $\epsilon_1-\epsilon_2$, $2\epsilon_2$, $2\epsilon_3$, and we give the highest weights, as well as the dimensions.

For this real form, there are three conjugacy classes of Cartan subgroups. The {\tt KGB} elements 0, 1, 12, 14 correspond to the maximally compact CSG $\mathbb T^2\times \mathbb C^{\times}$, elements 24 and 36 correspond to the maximally  split CSG 
$\mathbb C^{\times}\times (\mathbb R^{\times})^2$, and 
elements 3, 16 and 32 correspond to the intermediate CSG the is isomorphic to $(\mathbb C^{\times})^2$.

\medskip

\begin{tabular}{| c |  c | c | c | c | c|}\hline
 \#    & {\tt x} & $\lambda$ & $\nu$ & highest weight & dimension \\ \hline
 0  &   36 & (3,2,0,0) & (1,1,0,0) & (0,0,2) & 3 \\
 1 &  36 & (4,3,0,0) & (1,1,0,0) & (0,0,0)& 1 \\
 2 & 32 & (3,3,0,0) & (1,1,0,0) & (0,0,2) & 3 \\
 3 & 16 & (2,2,-1,0) & (1/2,1, -1/2,0) & (0,0,4) & 5 \\
 4 &24 & (1,2,1,0) & (0,1,0,0) & (1,1,0) & 5 \\
 5 & 24 & (1,3,1,0) & (0,1,0,0) & (1,1,2) & 15 \\
   &.     &.               &.               & (2,0,0) & 10\\
 6 & 14 & (1,2,0,0) & (0,1,0,0) & (2,0,0) & 10 \\
 7 & 12 & (1,2,0,0) & (0,1,0,0) & (1,1,2) &15 \\
 8 & 3 & (1,1,0,0) & (0,1/2,-1/2,0) & (2,0,2) & 30 \\
 9 & 1 & (1,1,0,0) & (0,0,0,0) & (0,0,6) & 7 \\
 10 & 0 & (1,1,0,0) & ( 0,0,0,0) & (1,1,4) & 25 \\ \hline
 
 \end{tabular}

\section{Acknowledgments} 
One of the preeminent researchers in the subject of minimal representations has been Toshiyuki Kobayashi \cite{HKM}. Thus it is our pleasure 
to dedicate this paper to Toshi on the occasion of his 60th birthday. In the process of writing this paper we benefited from discussions with Jeff Adams. In particular, 
we thank him for initial {\tt atlas} computations. 
G.S. would like to thank the National University of Singapore for hospitality during 2015 when this project and its large companion project \cite{GS22} were started. 
W.T.G. is partially supported by a Singapore government MOE Tier 1 grant R-146-000-320-114.
G.S. is partially supported by a Simons Collaboration Grant 946504. 


\end{document}